\documentclass[a4paper,oneside,reqno,svgnames]{amsart}
\usepackage[utf8x]{inputenc}
\usepackage[english]{babel}
\usepackage{lmodern} 

\usepackage{geometry}

\usepackage{amsmath,amsfonts,amssymb,amsthm}
\usepackage{mathtools}
\usepackage{bm} 			
\usepackage{stmaryrd} 		
\usepackage{bbm} 			
\usepackage{mathrsfs}   		
\usepackage{dsfont}  		

\usepackage{float}
\usepackage{tikz-cd}

\usepackage{enumitem} 		

\usepackage{xcolor}
\usepackage[colorlinks, linkcolor=DarkBlue, citecolor=cyan, backref=page]{hyperref} 
\usepackage{csquotes} 
\usepackage[textsize=tiny,textwidth=1.08in, color = SkyBlue]{todonotes}
\usepackage[bottom]{footmisc}

\theoremstyle{plain}
\newtheorem{theorem}{Theorem}[section]
\newtheorem{proposition}[theorem]{Proposition}
\newtheorem{lemma}[theorem]{Lemma}

\newtheorem*{theorem*}{Theorem}
\newtheorem*{proposition*}{Proposition}
\newtheorem*{lemma*}{Lemma}
\newtheorem*{corollary*}{Corollary}
\newtheorem*{property*}{Properties}
\newtheorem*{conjecture*}{Conjecture}

\theoremstyle{definition}
\newtheorem{definition}[theorem]{Definition}

\newtheorem{remark}[theorem]{Remark}
\newtheorem*{definition*}{Definition}
\newtheorem*{example*}{Example}
\newtheorem*{remark*}{Remark}

\theoremstyle{remark}

\newtheorem*{note*}{Remark}
\newtheorem*{exercise*}{Exercise}
\newtheorem*{notation*}{Notation}

\newcommand{\m}[1]{\mathcal{#1}}
\newcommand{\bb}[1]{\mathbb{#1}}

\newcommand{\mrm}[1]{\mathrm{#1}}
\newcommand{\mf}[1]{\mathfrak{#1}} 
\newcommand{\scr}[1]{\mathscr{#1}}

\newcommand{\del}{\partial}
\newcommand{\delbar}{\bar{\partial}} 
\newcommand{\dd}{\mrm{d}}

\newcommand{\Aut}{\mathrm{Aut}}
\newcommand{\Scal}{\mathrm{Scal}}
\newcommand{\Lie}{\mathrm{Lie}}
\newcommand{\grad}{\mathrm{grad}}
\newcommand{\ext}[1]{\bigwedge\nolimits^{#1}}

\newcommand{\Ok}[1]{O\left(k^{#1}\right)}

\DeclarePairedDelimiter\card{\lvert}{\rvert} 
\DeclarePairedDelimiter\norm{\lVert}{\rVert} 
\DeclarePairedDelimiter{\set}{\{}{\}}   

\numberwithin{equation}{section} 

\author{Annamaria Ortu}
\title{Moment maps and stability of holomorphic submersions}
\address{Department of Mathematical Sciences, University of Gothenburg, Sweden}
\email{ortu[at]chalmers[dot]se}

\begin{document}

\begin{abstract}
We prove a finite-dimensional moment map property for certain canonical relatively K\"ahler metrics on holomorphic fibrations, called \emph{optimal symplectic connections}. We then relate the existence of zeroes of this moment map to the \emph{stability} of the fibration, where the stability property we consider is a version of K-stability that takes into account the fibration structure, first introduced by Dervan--Sektnan.
In particular, we prove that a stable deformation of a fibration admitting an optimal symplectic connection still admits an optimal symplectic connection, through a new approach using the finite-dimensional moment map properties and the moment map flow.

We include an appendix with a proof of a result considered by Sz\'ekelyhidi that a K-polystable deformation of a constant scalar curvature K\"ahler manifold still admits a constant scalar curvature metric, using the same technique.
\end{abstract}

\maketitle
\thispagestyle{empty}

\section{Introduction}

Let $(M, \omega)$ be a symplectic manifold, and $K$ a compact group acting on it by symplectomorphisms. The action is said \emph{Hamiltonian} if there exists a map $\mu : M \to \Lie(K)^*$, called a \emph{moment map}, that is equivariant with respect to the action of $K$ on $M$ and the co-adjoint action on the dual Lie algebra, and that satisfies
\[
\dd_x \langle\mu,\xi \rangle = \omega(-, \sigma_x(\xi)),
\]
where $\sigma$ denotes the infinitesimal action. Since a Hamiltonian function is only unique up to a constant, the role of the moment map is to choose a Hamiltonian function for the infinitesimal vector fields.
When the manifold $M$ is projective and $\omega$ is in the first Chern class of a ample line bundle $L$, the notion of a moment map is related to the algebro-geometric notion of stability in the sense of Geometric Invariant Theory (GIT). Originally introduced by Mumford \cite{Mumford_Stability} to study the moduli space of holomorphic vector bundles, GIT stability is related to moment maps by the Kempf--Ness theorem \cite{KempfNess} which establishes an equivalence between the existence of zeroes of a moment map in a given orbit and the GIT polystability of the orbit.

Donaldson \cite{Donaldson_momentmap} and Fujiki \cite{Fujiki_momentmap} extended the theory of Hamiltonian actions to the infinite-dimensional setting of K\"ahler metrics and complex structures, and they proved that the constant scalar curvature operator is a moment map for the action of the group of Hamiltonian symplectomorphisms on the space of almost complex structures on a compact symplectic manifold.
The Yau-Tian-Donaldson \cite{Yau_OpenProblems, Tian_KahlerEinstein, Donaldson_1StabilityToric} conjecture is an infinite-dimensional analogue of the Kempf--Ness theorem: it predicts that the existence of a constant scalar curvature K\"ahler (cscK) metric is equivalent to an algebro-geometric notion of stability, \emph{K-stability}, which is not a genuine GIT notion but parallels GIT stability in the infinite-dimensional setting.
While still open in full generality, the conjecture is known to be true for Fano varieties, due to Chen--Donaldson--Sun \cite{ChenDonaldsonSun_Fano1, ChenDonaldsonSun_Fano2, ChenDonaldsonSun_Fano3}.
The fact that the existence of constant scalar curvature K\"ahler (cscK) metrics implies K-stability is also a theorem of Donaldson \cite{Donaldson_1StabilityToric}, Stoppa \cite{Stoppa_Kstability} and Berman--Darvas--Lu \cite{BermanDarvasLu_Kenergy}.

A finite-dimensional result in the opposite direction was first considered by Sz\'ekelyhidi \cite{Szekelyhidi_deformations}: a K-polystable infinitesimal deformation of a cscK manifold still admits a cscK metric.
Sz\'ekelyhidi's technique uses that the scalar curvature function is a moment map on the finite-dimensional vector space that parametrises infinitesimal deformations of a cscK manifold.
There is a gap, however, in Sz\'ekelyhidi's proof, explained in Remark  \ref{rmk:error}.

In this paper, we prove a moment map property and we establish an analogous result in the setting of holomorphic fibrations.
Instead of cscK metrics, on the fibration we consider a canonical choice of a relatively K\"ahler metric, called an \emph{optimal symplectic connection} \cite{DervanSektnan_OSC1, Ortu_OSCdeformations}.
Our main result is the following.

\begin{theorem}\label{Thm1}
Let $Y \to B$ be a proper holomorphic submersion admitting an optimal symplectic connection and let $W \to B$ be a deformation. If $W \to B$ is stable, then it admits an optimal symplectic connection.
\end{theorem}

We use a new technique to prove this: it consists in reducing the infinite-dimensional problem of finding a solution to a PDE to a finite-dimensional problem of finding a zero of a moment map on the Kuranishi space, for which we rely on the recent theory of Dervan--Hallam  \cite{DervanHallam}.
Then we use the stability condition and the moment map flow to find an actual solution.

A similar approach, using moment map flows, was recently used in a different class of problems by Dervan \cite{Dervan_StabilityVarieties}, Delloque \cite{delloque2024HYM}, and the author with Sektnan \cite{OrtuSektnan_ssvb}.
One key novelty in the proof of Theorem \ref{Thm1} is that we need to prove a ``stability" result for the moment map flow showing that the flow remains in a bounded region inside the Kuranishi space (Proposition \ref{prop:stabilityflow}).
This is what allows us to use the moment map flow uniformly across a neighbourhood of the origin in the Kuranishi space.

In an appendix to this paper, we use the same approach to prove the following result on K-polystable deformations of cscK manifolds, already mentioned above. This fixes a mistake in Sz\'ekelyhidi's discussion \cite[Proposition 8]{Szekelyhidi_deformations}.

\begin{theorem}\label{Thm:App}
Let $X$ be a cscK manifold and $Y$ be a K-polystable deformation of $X$. Then $Y$ admits a cscK metric.
\end{theorem}

We next briefly explain the optimal symplectic connection condition and the notion of stability we employ.
Let $Y \to B$ be a proper holomorphic submersion and let us fix an ample line bundle $L$ on the base and a relatively ample line bundle $H_Y$ over $Y$. The base $(B, L)$ and the relative polarisation $H_Y$ are considered fixed throughout.
To define optimal symplectic connections, we need to impose the following condition on the fibres, given in terms of analytic K-stability: we require that the submersion $(Y, H_Y) \to (B,L)$ degenerates to a submersion $(X, H_X) \to (B,L)$ whose fibres admit a cscK metric.
A two-form $\omega \in c_1(H_Y)$ induces a splitting of the tangent bundle of $Y$ into a vertical and a horizontal part, defined by orthogonality with respect of $\omega$, and is thus called a symplectic connection. An \emph{optimal} symplectic connection is defined as the solution to the geometric partial differential equation
\begin{equation}\label{Eq:OSCintro}
p_E(\Delta_{\m{V}}(\Lambda_{\omega_B} (\gamma^*F_{\m{H}})) + \Lambda_{\omega_B} \rho_{\m{H}}) + \frac{\lambda}{2}\nu = 0.
\end{equation}
In this expression, $F_{\m{H}}$, $\rho_{\m{H}}$ and $\nu$ are curvature quantities which depend on $\omega$ and $\lambda >0$ is a constant. The map $p_E$ is the projection onto the global sections of the vector bundle $E \to B$ of fibrewise holomorphy potentials with respect to the relatively cscK complex structure of $X$.

Optimal symplectic connections were first introduced by Dervan--Sektnan \cite{DervanSektnan_OSC1} in the case when the fibres admit a constant scalar curvature metric, and generalised by the author \cite{Ortu_OSCdeformations} in the relatively (i.e.\ fibrewise) K-semistable case. The condition reduces to the Hermite--Einstein equation when the fibration is defined as the projectivisation of a holomorphic vector bundle.

In \cite{Ortu_OSCmoduli} the author has developed an analytic deformation theory for optimal symplectic connections.
The objective there was to construct an analytic moduli space of holomorphic fibrations with discrete automorphisms.
Given a fibration $Y \to B$ admitting an optimal symplectic connection, there exists a complex analytic space, denoted by $V_\pi^+$, that parametrises the deformations of the complex structure of $Y \to B$ that are compatible with the relatively symplectic form and that preserve the projection onto the base and the property of degenerating to relatively cscK fibrations. In the study of deformations, it is crucial to allow K-semistable fibres, since the relatively cscK condition is not open.

Theorem \ref{Thm1} answers the question of when a deformation of the type just described still admits an optimal symplectic connection.
In particular, we relate the optimal symplectic connection property to \emph{stability of fibrations}, introduced by Dervan--Sektnan \cite{DervanSektnan_OSC2} and modelled on K-stability. They introduced a notion of fibration degeneration that extends the one of test configuration to account for the fibration structure, and they define stability in terms of a numerical invariant derived from the Donaldson-Futaki invariant.
In parallel with the Yau-Tian-Donaldson conjecture and the Hitchin-Kobayashi correspondence,
Dervan--Sektnan predicted that a holomorphic submersion is stable if and only if it admits an optimal symplectic connection, and they proved that the existence of a solution implies (semi)stability.
Their result, and the definition of stability itself, were further improved by Hallam \cite{Hallam_geodesics}.
Theorem \ref{Thm1} provides the first result proving that the existence of a optimal symplectic connection follows from the stability of the fibration.

We finally describe in more detail the technique of the proof of Theorem \ref{Thm1} in the context of optimal symplectic connections: it consists of three steps.
First, we perturb the relatively K\"ahler metric so that the optimal symplectic connection operator lies in the Lie algebra of the reductive group $K_{\pi, v}^{\bb{C}}$ of automorphisms of the equation; this uses the linearisation of the equation and the implicit function theorem.

Then, we prove that the optimal symplectic connection operator is a moment map on $V_\pi^+$ for the action of $K_{\pi, v}^{\bb{C}}$.
The Dervan--Hallam approach \cite{DervanHallam}, which we use at this step, consists of defining the moment map through the universal deformation family.
The advantage is that, by changing at the same time the K\"ahler metric and the complex structure, we obtain an interpretation of the scalar curvature as a finite-dimensional moment map with respect to the \emph{Weil--Petersson K\"ahler metric}.
This is in contrast with the more classical approaches to perturbative problems, which fix the K\"ahler form and vary the complex structure, and thus produce a finite-dimensional moment map for the scalar curvature with respect to a symplectic form which is not K\"ahler.

The final step of the proof is to use the \emph{moment map flow} for the finite-dimensional moment map; this is where the stability assumption is used.
In particular, having a K\"ahler moment map is crucial for applying the moment map flow.
Frequently used in the symplectic approach to GIT, the moment map flow is the flow along the gradient of the moment map squared; we mostly follow the formultation of Dervan--McCarthy--Sektnan \cite[\S4.2]{DervanMcCarthySektnan}, which in turn builds on \cite{GeorgoulasRobbinSalamon_GITbook, ChenSun_CalabiFlow}, and we extend these works to apply to the case of non-trivial stabiliser.
The flow starting from the point $w \in V_\pi^+$ that corresponds to a stable fibration $W \to B$ converges to a point $w_{\infty}$ inside $V_\pi^+$, representing a deformation of the fibration; this is where we use our ``stability" result for the moment map flow.
We prove that the stability of $W \to B$ implies that $w_\infty$ is a zero of the associated moment map and it belongs to the $K_{\pi, v}^{\bb{C}}$-orbit of $w$. Hence it is a solution to the optimal symplectic connection equation.

Our proof of Theorem \ref{Thm:App} consists of the same three steps, but applied to the technically simpler case of cscK metrics.

\subsection*{Outline}
In Section \ref{Sec:background} we review the moment map flow and K-stability. We also describe the theory of deformations of cscK metric in the context of varying the K\"ahler form and the complex structure at the same time.
In Section \ref{Sec:OSC}, we review the theory of deformations of fibrations and of optimal symplectic connections and we prove a Matsushima-type criterion for the group of automorphisms of the optimal symplectic connection equation.
In Section \ref{Sec:momentmap} we prove that the optimal symplectic connection operator is a moment map for the action of said group on the finite-dimensional space of infinitesimal deformations and we prove Theorem \ref{Thm1}.
In Appendix \ref{appendix} we give a proof that a K-polystable deformation of a cscK manifold admits a cscK metric.

\subsection*{Acknowledgments}
I wish to thank Vestislav Apostolov, Ruadha\'i Dervan, Michael Hallam, Lars Martin Sektnan, and Jacopo Stoppa for many discussions on the subject of this paper.
I thank Andr\'es Ib\'a\~{n}ez N\'u\~nez for discussions related to the moment map flow and Remark \ref{rmk:flow_stays_inside}.
I thank the INdAM ``National Group for Algebraic and Geometric Structures, and their Applications" (GNSAGA) for financial support. I thank the Isaac Newton Institute for Mathematical Sciences, Cambridge, for support and hospitality during the programme \textit{New equivariant methods in algebraic and differential geometry} where work on this paper was undertaken. This work was supported by EPSRC grant no EP/R014604/1.

\section{Preliminaries}\label{Sec:background}
In this section, we recall some definitions and results on the moment map flow, on K-stability and on deformations of K\"ahler metrics with constant scalar curvature.

Let $(M, L)$ be a polarised K\"ahler manifold with a fixed ample line bundle $L$, $\omega$ a K\"ahler form in the first Chern class of $L$ and let $J$ be the complex structure of $M$. We denote by $g = g(\omega,J)$ the Riemannian metric on $M$ induced by $J$ and $\omega$, i.e.\
\[
g(\cdot, \cdot) = \omega(\cdot, J\cdot).
\]

The \emph{scalar curvature} of the K\"ahler metric $g(\omega,J)$ is a smooth function on $M$ defined as the contraction of the Ricci curvature:
\[
\mrm{Scal}(\omega,J) := \Lambda_{\omega} \mrm{Ric}(\omega, J).
\]

We consider K\"ahler metrics with constant scalar curvature, where the constant is given by the intersection product
\[
\widehat{S} = \frac{n \ c_1(M) \cdot c_1(L)^{n-1}}{c_1(L)^n}.
\]
In particular, $\widehat{S}$ is a \emph{topological constant} fixed by the polarisation.

Let $\Aut(M, L)$ be the group of automorphisms of $M$ which lift to $L$ and let $\mf{h}_0$ be its Lie algebra.
Then $\mf{h}_0$ can be characterised as the space of holomorphic vector fields that can be written as the Riemannian gradient $\nabla_g f$, for some function $f$ called a \emph{holomorphy potential} for the holomorphic vector field.
Let $K=\mrm{Isom}(M, \omega)$ be the group of holomorphic isometries of the K\"ahler metric $(\omega, J)$ and let $\mf{k}$ be its Lie algebra.
A well-known result of Matsushima and Lichnerowicz, known as the \emph{Matsushima criterion} or the \emph{Cartan decomposition}, states that when $\omega$ is cscK, the group $\Aut(M,L)$ is reductive \cite{Matsushima_structure_KE, Lichnerowicz_geometrie_groups_transformations} (see also \cite[\S3.4]{Gauduchon_Extremal_Introduction}).

\begin{theorem}\label{Thm:Matsushima_criterion}
Suppose that there exists a constant scalar curvature K\"ahler metric in $c_1(L)$. Then
\begin{equation*}
\mf{h}_0 = \mf{k}_0 \oplus J\mf{k}_0,
\end{equation*}
i.e.\ the group $\Aut(M,L)$ is reductive.
\end{theorem}

\subsection{The moment map flow}\label{subsec:momentmapflow}
The moment map flow is the geometric flow along the gradient of the norm squared of the moment map. It appears frequently in differential-geometric approaches to GIT \cite{Donaldson_scalarcurvatureprojembeddings,ChenSun_CalabiFlow,GeorgoulasRobbinSalamon_GITbook,DervanMcCarthySektnan}.
We explain some properties of existence and stability that will be used in the following sections.

Let $(V, \omega)$ be a K\"ahler manifold, which we do not assume is compact.
A vector field $\eta$ is \emph{Hamiltonian} with respect to $\omega$ if there exists a function $h \in C^\infty(V, \bb{R})$ such that 
\[
\omega (\eta, \cdot)= -\dd h.
\]
We say that $h$ is the \emph{Hamiltonian function} of $\eta$.
On a K\"ahler manifold $\eta = J\nabla_g(h)$, where $\nabla_gh$ is the Riemannian gradient of $h$.
A Hamiltonian vector field with Hamiltonian $h$ is also called the \emph{symplectic gradient} of $h$, and denoted by $\mrm{grad}^\omega h$.

\begin{definition}
Let $K$ be a Lie group that acts on $V$ by symplectomorphisms, i.e.\ for any $g \in K$, $g^*\omega = \omega$.
Let $\mf{k}$ be the Lie algebra of $K$.
For any element $\xi \in \mf{k}$, the \emph{infinitesimal action} of $\xi$ is the vector field
\[
\sigma_x(\xi) = \left.\frac{\dd}{\dd t}\right\vert_{t=0} \left( \mrm{exp}(-t\xi)\cdot x\right).
\]
We say that the action is Hamiltonian if there exists a \emph{moment map}
\[
\mu : V \to \mf{k}^*
\]
that is equivariant with respect to the $K$-action on $V$ and the co-adjoint $K$-action on the dual Lie algebra $\mf{k}^*$ and such that for each $x \in V$
\[
\dd_x \langle\mu,\xi \rangle = \omega(\cdot, \sigma_x(\xi)),
\]
i.e.\ $\langle\mu,\xi\rangle$ is a Hamiltonian function for the vector field $\sigma(\xi)$ on $V$.
We often use the notation $\sigma_\xi$ for $\sigma(\xi)$.
\end{definition}
It is clear from the definition of a Hamiltonian vector field that the Hamiltonian function is only unique up to a constant.
The moment map then chooses a Hamiltonian function for the infinitesimal vector field.
We next fix a bi-invariant inner product on $\mf{k}$ with rational coefficients on the centre of $\mf{k}$, and we use it to identify $\mf{k}$ and $\mf{k}^*$; therefore, we view the moment map as a map with values in $\mf{k}$. 

\begin{definition}\label{def:momentmapflow}
Let $x \in V$.
The \emph{moment map flow} associated with the moment map $\mu$ with starting point $x$ is
\[
\frac{\dd}{\dd t} x_t = J\sigma(\mu(x_t))
\]
where $\sigma(\mu(x_t))$ is the infinitesimal vector field associated to the function $\mu(x_t)$ and $J$ is the almost complex structure of $V$.
\end{definition}
In particular, the moment map flow is the negative gradient flow of the norm squared of the moment map, i.e.
\[
J\sigma(\mu(x)) = -\frac{1}{2}\nabla \norm{\mu(x)}^2.
\]
We denote the flow starting at $x$ as $\phi_t(x)$.

\begin{proposition}[Stability of the flow]\label{prop:stabilityflow}
Assume that $\mu(0)=0$. Then there exists a neighbourhood $N(0) \subset V$ such that for all $x \in N(0) \subset V$ the moment map flow starting at $x$ remains in $V$.
\end{proposition}
\begin{proof}
Let us view $V$ as an open ball inside a projective space $\bb{P}^{d}$: if $[x_0: \dots: x_d]$ are homogeneous coordinates on $\bb{P}^{d}$, the space $V$ can be viewed as an open ball around the origin in the affine open subset $x_0=0$.
Following \cite[Lemma 4.7]{DervanMcCarthySektnan}, we can extend the action of $K$, the K\"ahler form $\omega$, and the moment map $\mu$ to $\bb{P}^d$ in such a way that their restriction to $V$ agrees with the original data.
This allows us to appeal to results proven in the compact setting.
Let $\mu_{\bb{P}^d}$ be the extended moment map.
The moment map flow associated with $\mu_{\bb{P}^d}$ exists for all time. Moreover \cite[Theorem 4.16]{Kirwan_cohomologybook} the limit set of the flow is a disjoint union of path-connected closed subsets where the norm square of the moment map has constant value. Let $C$ be one connected component and let $S_C$ be the stratum of points in $\bb{P}^d$ whose flow ends in $C$.
A result of Duistermaat \cite{Lerman_gradientflowmomentmap} guarantees that the flow starting at $x\in S_C$ converges to a single point $x_\infty$ and the map
\[
\begin{aligned}
[0,+\infty] \times S_C &\to S_C\\
(t,x) &\mapsto \phi_t(x)
\end{aligned}
\]
is a deformation retraction. In particular, the map
\begin{equation}\label{Eq:limit_flow_continuous}
x \mapsto x_\infty
\end{equation}
is continuous.
Let now $C$ be the connected component of the limit set that contains the origin of $V$; $S_C$ is open because the origin is a zero of the moment map, hence in particular is GIT-semistable by the Kempf--Ness theorem \cite{KempfNess}, and semistability is an open condition. Since $V$ is open there exists a neighbourhood $N'(0) \subset C$ contained in $V$.
From the continuity of the map \eqref{Eq:limit_flow_continuous}, we can find a neighbourhood $N(0)$ in $V$ such that for each $x \in N(0)$ the limit point $x_\infty \in N'(0)$.
\end{proof}

Let us fix the starting point $x \in V \cap N(0)$.
We next explain how to restrict the flow by projecting orthogonally to the stabiliser group of $x$.
Let $K_x$ be the stabiliser of $x$ and let $T$ be a maximal torus in the stabiliser, with $\mf{t}$ its Lie algebra.
Let
\begin{equation}\label{eq:Lie_algebra_orthogonal}
\mf{k}_{T^\perp} = \set{\xi \in \mf{k} \mid \langle \xi, h\rangle = 0 \ \text{for all}\ h \in \mf{t}}.
\end{equation}
In particular, one can prove that, assuming $\langle \cdot, \cdot\rangle$ is a rational inner product, the Lie algebra $\mf{k}_{T^\perp}$ corresponds to a subgroup $K_{T^\perp} < K$ and the same is true for its complexification \cite[\S5.5]{Szekelyhidi_book}.
Thus we can define a moment map for the action of $K_{T^\perp}$ by projecting
\[
\mu_{T^\perp} : V \to \mf{k} \to \mf{k}_{T^\perp},
\]
and we can consider the moment map flow associated with the moment map $\mu_{T^\perp}$ with starting point $x$:
\[
\frac{\dd}{\dd t} x_t = J\sigma(\mu_{T^\perp}(x_t))
\]
where $\sigma(\mu_{T^\perp}(x_t))$ is the infinitesimal vector field associated to the function $\mu_{T^\perp}(x_t)$.

\begin{proposition}[{\cite[Theorem 3.3]{GeorgoulasRobbinSalamon_GITbook}}]\label{Prop:existence_limit_flow}
The limit point $x_\infty \in V$ of the moment map flow starting at $x$ belongs to the closure of the orbit $K_{T^\perp}^{\bb{C}} \cdot x$ and  $\mu_{T^\perp}(x_\infty)$ belongs to the Lie algebra $(\mf{k}_{T^\perp})_{x_\infty}$ of the stabiliser of $x_\infty$ under the action of $K_{T^\perp}$.
\end{proposition}

Moreover, for the action of $K_{T^\perp}$, the point $x$ has now discrete stabiliser, and we can state the following result \cite[Corollary 4.14]{DervanMcCarthySektnan}.
\begin{proposition}\label{Prop:not_in_orbit}
Let $x_\infty$ be the limit of the moment map flow starting at $x$. Then exactly one of the following holds:
\begin{enumerate}
\item \label{List:theta0} $\mu_{T^\perp}(x_\infty)= 0$ and $x_\infty \in K_{T^\perp}^{\bb{C}}\cdot x$;
\item \label{List:thetanegative} $x_\infty \notin K_{T^\perp}^{\bb{C}}\cdot x$ and there exists a point $\widetilde{x}\in V$ and an element $\xi \ne 0$ in $ \mf{k}_{T^\perp}$ such that
\[
\lim_{t\to \infty} \exp(-it\xi) \cdot x = \widetilde x
\]
and $\langle\mu_{T^\perp} (\widetilde{x}), \xi\rangle\ge 0$.
Further, when $\langle\mu_{T^\perp} (\widetilde{x}), \xi\rangle =0$, then $\widetilde x=x_\infty$ and  $\mu_{T^\perp}(x_\infty)= 0$.
\end{enumerate}
\end{proposition}

\subsection{K-stability}
We recall the notion of K-stability, introduced by Tian \cite{Tian_K-stability_Fano} and Donaldson \cite{Donaldson_1StabilityToric} to be an analogue for polarised varieties to the Hilbert-Mumford criterion for GIT.
\begin{definition}
Let $(X,L)$ be a polarised variety of dimension $n$. A test configuration for $(X,L)$ is the data of
\begin{enumerate}
\item a variety $\m{X}$ with a $\bb{C}^*$-equivariant flat proper morphism to $\bb{C}$;
\item a relatively ample line bundle $\m{L} \to \m{X}$ together with a lift of the $\bb{C}^*$-action to it;
\item an isomorphism $(\m{X}_1, \m{L}_1) \simeq (X, L^r)$ for some $r>0$.
\end{enumerate}
We say that $(\m{X},\m{L})$ is a product test configuration if $(\m{X}, \m{L}) \simeq (X,L^r) \times \bb{C}$, with a possibly nontrivial $\bb{C}^*$-action, and is a trivial test configuration if $(\m{X}, \m{L}) \simeq (X,L) \times \bb{C}$ with trivial $\bb{C}^*$-action.
\end{definition}
Given a test configuration, one associates a numerical invariant as follows. Consider the following expansions for the dimension of $H^0(\m{X}, \m{L}_0^j)$ and for the weight $w_j$ of the induced action of $\bb{C}^*$ on $H^0(\m{X}, \m{L}_0^j)$:
\[
\begin{aligned}
\dim H^0(\m{X}, \m{L}_0^j) &= a_0j^n + a_{1}j^{n-1} + O\left(j^{n-2}\right)\\
w_j &= b_0 j^{n+1} + b_1j^{n} + O\left(j^{n-1}\right).
\end{aligned}
\]
The \emph{Donaldson-Futaki} invariant, introduced by Donaldson in \cite{Donaldson_1StabilityToric}, is the number
\[
\mrm{DF}(\m{X}, \m{L}) := \frac{a_1b_0-a_0b_1}{a_0^2}.
\]
The following result \cite{Donaldson_1StabilityToric} (see also \cite{Legendre_localizingDFinvariant} for a generalisation) relates the Donaldson-Futaki invariant with the classical Futaki invariant
\[
\mrm{Fut}(\sigma_h) = \int_X (\Scal(\omega)-\widehat{S})h \omega^n,
\]
where $\sigma_h \in \mf{h}_0$ and $h$ is its holomorphy potential.
\begin{theorem}\label{Thm:classical_Futaki_equal_DF}
When the central fibre is smooth, the Donaldson-Futaki invariant satisfies
\[
\frac{\mrm{DF}(\m{X}, \m{L})}{n!} = -\pi \mrm{Fut}(\sigma_h),
\]
where $\sigma_h$ is the holomorphic vector field generating the $\bb{C}^*$-action.
\end{theorem}
If $(\m{X}, \m{L})$ is a test configuration, one can define the \emph{normalization} of $(\m{X}, \m{L})$ by taking the normalization $\widetilde{\m{X}}$ of $\m{X}$ and the pullback of $\m{L}$ to $\widetilde{\m{X}}$. The normalization is again a test configuration for $(X, L)$ \cite[\S 5]{RossThomas_HilbertMumford}.
\begin{definition}[\cite{Donaldson_1StabilityToric,Tian_KahlerEinstein}]
A polarised variety $(X, L)$ is
\begin{enumerate}
\item \emph{K-semistable} if $\mrm{DF}(\m{X}, \m{L}) \ge 0$ for all test configurations $(\m{X}, \m{L})$ for $(X, L)$;
\item \emph{K-polystable} if it is K-semistable and $\mrm{DF}(\m{X}, \m{L})$ vanishes only if $(\m{X}, \m{L})$ normalises to a product test configuration;
\item \emph{K-stable} if it is K-semistable and $\mrm{DF}(\m{X}, \m{L})$ vanishes only if $(\m{X}, \m{L})$ normalises to the trivial test configuration.
\end{enumerate}
\end{definition}

\subsection{Deformation theory of K\"ahler metrics with constant scalar curvature}
\label{Sec:Sz_def_theory}
The deformation theory of cscK manifolds was developed by Sz{\'e}kelyhidi \cite{Szekelyhidi_deformations}, Br\"onnle \cite{Bronnle_PhDthesis}, and more recently by Dervan--Hallam \cite{DervanHallam} with a different approach; in this section we explain and combine the two approaches.

Let $(M, L)$ be a polarised K\"ahler manifold and let $\omega$ be a fixed K\"ahler form in $c_1(L)$.
Define the infinite-dimensional manifold $\scr{J}$ of almost complex structures compatible with $\omega$.
The tangent space at a point $J \in \scr{J}$ can be identified with
\[
T_J^{0,1}\scr{J} = \set*{\left. \alpha \in \Omega^{0,1}(T^{1,0}M)\right\vert \omega(\alpha(u), v) + \omega(u, \alpha(x)) = 0}.
\]

Let $\scr{G}$ be the group of Hamiltonian symplectomorphisms of $(M, \omega)$. Then $\scr{G}$ acts on $\scr{J}$ by pull-back.
The Lie algebra of $\scr{G}$ can be identified with the space of smooth functions on $M$ with mean value zero, denoted $C_0^\infty(M, \bb{R})$.

We fix $J \in \scr{J}$ an integrable complex structure on $(M, \omega)$ such that $\Scal(\omega, J)$ is constant.
The infinitesimal action of $\scr{G}$ on $\scr{J}$ is defined by the operator
\begin{equation}\label{Eq:map_P}
\begin{aligned}
P : C_0^\infty(M, \bb{R}) &\to T_J\scr{J} \\
h &\mapsto \m{L}_{\eta_h} J,
\end{aligned}
\end{equation}
where $\eta_h$ is the Hamiltonian vector field with function $h$; let $P^{\bb{C}}$ be its extension to $C_0^\infty(M, \bb{C})$.

The deformations of the complex structure $J$ which are integrable to first-order are parametrised by the finite-dimensional vector space
\begin{equation}\label{Eq:H_tilde_cscK}
\widetilde{H}^1 = \ker \left(PP^*+(\delbar^*\delbar)^2\right)
\end{equation}
on $T_{J}\scr{J}$.
Consider the group of Hamiltonian isometries of $(M, \omega, J)$, denoted by $K$: it is the stabiliser of the complex structure $J$ for the action of $\scr{G}$ which means that, by definition, it is the intersection of $\scr{G}$ with $\mrm{Aut}(M, J)$.

The group $K$ can be complexified and from Matsushima's criterion (Theorem \ref{Thm:Matsushima_criterion}) its complexification is $\mrm{Aut}(M, L)$.
In particular, we can realise the Lie algebra of $K$, denoted $\mf{k}$, as the space of real holomorphy potentials with respect to $(\omega, J)$ with mean-value zero; it follows that the Lie algebra of $K$ can be identified with the kernel of $P$.
Changing the K\"ahler metric $\omega$ to $\omega+ i\del\delbar\varphi$ changes a holomorphy potential $h$ to
\[
h + \langle \nabla h, \nabla \varphi \rangle,
\]
see \cite[Lemma 4.10]{Szekelyhidi_book}.
Given $\xi \in \mf{k}$, we denote by $\sigma_\xi$ the corresponding infinitesimal vector field and by $h_\xi$ the corresponding holomorphy potential. 

The following theorem describes the deformations of the cscK complex structure $J$ and their scalar curvature.
In the following, when $J_x\in \scr{J}$ is a non-integrable almost complex structure, we denote by $g_x$ the Hermitian metric induced by $(\omega, J_x)$ and by $\Scal(\omega, J_x)$ its Chern-scalar curvature.
Moreover, for a function $\varphi$ in the Sobolev space $W^{2, p}(g_x)$ such that $\norm{\varphi}_{W^{2,p}} \ll 1$, let
\[
\omega_\varphi := \omega +\frac{1}{2} (\dd\dd^c\varphi)^{1,1},
\]
where the $(1,1)$-part is taken with respect to $J_x$ and $\dd\dd^c = -\dd J_x  \dd$.
While in general $\omega_\varphi$ is not closed, it satisfies the compatibility conditions
\[
\begin{aligned}
\omega_\varphi (J_x \cdot, J_x \cdot) = \omega_\varphi(\cdot,\cdot),\\
\omega_\varphi (\cdot, J_x \cdot) >0,
\end{aligned}
\]
thus $(\omega_\varphi, J_x)$ induces a Hermitian metric.
When $J_x$ is integrable, then $\omega_\varphi$ is closed and coincides with $\omega + i\del\delbar\varphi$.

\begin{theorem}\label{Thm:Kuranishi}
There exists a ball around the origin $V \subset \widetilde{H}^1$ and a $K$-equivariant holomorphic map
\begin{equation}\label{Eq:Kuranishi_map}
\Phi : V \to \scr{J}
\end{equation}
such that $\Phi (0) = J$ and
\begin{enumerate}
\item if two points $x$ and $x'$ of $V$ are in the same orbit for the complexified action of $K$, and $\Phi(x)$ is integrable, then their images $\Phi(x)$ and $\Phi(x')$ are isomorphic;
\item for every $x\in V$ there exists a smooth function $\varphi_x$ on $M$ and $\xi \in \mf{k}$ such that
\[
\Scal(\omega_{\varphi_x}, \Phi(x)) - \widehat{S} =  h_\xi + \langle \nabla_{g_x} h_\xi, \nabla_{g_x} \varphi_x\rangle_{g_x},
\]
where $\varphi_x$ varies smoothly with $x$.
\end{enumerate}
\end{theorem}
\begin{proof}
The first claim is mainly due to Kuranishi \cite{Kuranishi_family_cpx_str}; a proof adapted to take into account the compatibility with the symplectic form can be found in \cite[\S6]{ChenSun_CalabiFlow}.
We prove the second claim. For each $x \in V$, let $g_x$ be the Hermitian metric induced by $\omega$ and $\Phi(x)$. Consider the operator

\[
\begin{aligned}
G_x:\mf{k} \times W^{2,p+4}(g_x) &\to W^{2,p}(g_x)\\
(\xi, \varphi) &\mapsto\Scal(\omega_\varphi, \Phi(x)) - \widehat{S} - h_\xi - \langle \nabla_{g_x} h_\xi, \nabla_{g_x} \varphi \rangle_{g_x}.
\end{aligned}
\]
By our assumption, $G_0(0,0)=0$.
To compute the linearisation of $G_x$ at $(0, 0)$ in the $\varphi$ variable we use the expression of the first variation of the Chern scalar curvature \cite[Proposition 3.6]{AngellaPediconi_ChernScal} in the direction of $(\dd\dd^c\varphi)^{1,1}$.
We first observe that, when $\dd^c = -\Phi(x) \circ \dd$, i.e.\ when we consider the almost complex structure $\Phi(x)$, then
\[
\dd\dd^c\varphi = (\dd\dd^c\varphi)^{1,1} + \frac{1}{2}\dd\varphi \circ N_x,
\]
where $N_x$ is the Nijenhuis tensor of $\Phi(x)$.
Then the linearisation of $G_x$ is the operator\footnote{More precisely, this is computed by applying \cite[Proposition 3.6]{AngellaPediconi_ChernScal} to the variation given by $\left(-(\dd\dd^c\varphi)^{1,1}\right)^{\sharp_{\omega}}$, where $\sharp_{\omega}$ denotes the operation of raising an index using $\omega$.}
\[
\begin{split}
DG_x \vert_{(0,0)} (\varphi) = &-\Delta_{g_x}^2 (\varphi) - \langle \mrm{Ric}_x, (\dd\dd^c\varphi)^{1,1} \rangle_{g_x} \\
&+ \Delta_{g_x} \left(-\langle \dd\varphi, \vartheta_x \rangle_{g_x} + \mrm{Tr}_{g_x}\left(\frac{1}{2}\dd\varphi \circ N_x\right)\right) + \langle \mrm{Ric}_x,\frac{1}{2}\dd\varphi\circ N_x\rangle_{g_x} \\
&+ \langle \dd (-\Delta_g\varphi +  \langle\dd\varphi, \vartheta_x \rangle_{g_x}) + \dd \mrm{Tr}_{g_x} \left(\frac{1}{2}\dd\varphi \circ N_x\right), \vartheta_x \rangle_{g_x},
\end{split}
\]
where $\vartheta_x$ is the torsion 1-form of the Chern connection of the metric $g_x$ and $\mrm{Ric}_x$ is the Chern-Ricci form of $g_x$.
Moreover,
\[
-\Delta_{g_x}^2 (\varphi) - \langle \mrm{Ric}_x, (\dd\dd^c\varphi)^{1,1} \rangle_{g_x} = -\m{D}^*_x\m{D}_x (\varphi) + \langle \nabla_{g_x}\mrm{Scal}(\omega, \Phi(x)), \nabla_{g_x}\varphi \rangle_{g_x},
\]
which is the linearisation of the scalar curvature when $(\omega, \Phi(x))$ is K\"ahler. Consider the Lichnerowicz operator
\[
L_x (\varphi) = -\m{D}^*_x\m{D}_x (\varphi) .
\]
The kernel of $L_x$ is contained in $\ker \m{D}^*_0\m{D}_0 = \mf{k}$, hence $L_x$ is surjective. 

Next, we show that its right inverse $Q_x$ satisfies
\begin{equation}\label{eq:rinverse_lich}
\norm{Q_x}< C
\end{equation}
for some constant $C$.
Let $\lambda_1$ be the first eigenvalue of the Laplacian $\Delta_0$ with respect to the Riemannian metric $g_0=g(\omega, \Phi(0))$; it satisfies
\[
\lambda_1 = \inf_{\varphi \in \mf{k}^\perp} \frac{\langle \Delta_0(\varphi), \varphi\rangle_{g_0} }{\norm{\varphi}^2_{g_0}},
\]
which is equivalent to the inequality
\[
\norm{\delbar_0\varphi}^2_{L^2(g_0)} \ge \lambda_1 \norm{\varphi}^2_{L^2(g_0)}, \quad \varphi \in \mf{k}^\perp.
\]
Then we can write
\[
\langle \Delta_x^{g_0}(\varphi), \varphi\rangle_{g_0} = \norm{\delbar_x (\varphi)}^2_{g_0} + \norm{\delbar^*_x (\varphi)}^2_{g_0},
\]
where the $\delbar$-operator is with respect to the complex structure $\Phi(x)$ but the adjoint operator is with respect to the metric $g_0$. Since $\delbar_x$ is continuous in $x$ it follows that
\[
\inf_{\varphi \in \mf{k}^\perp} \frac{\langle \Delta_x^{g_0}(\varphi), \varphi\rangle_{g_0} }{\norm{\varphi}^2_{g_0}} \ge \frac{\lambda_1}{2}.
\]
Since the metrics $g_x$ and $g_0$ are uniformly equivalent when $x$ is in a sufficiently small neighbourhood of the origin, we obtain the inequality
\[
\norm{\delbar_x\varphi}^2_{L^2(g_x)} \ge C \norm{\varphi}^2_{L^2(g_x)},
\]
with constant independent on $x$. This in turn gives a Poincar\'e type inequality for $\m{D}_x$ with constant independent on $x$:
\[
\norm{\m{D}_x \varphi}^2_{L^2(g_x)} \ge C \norm{\nabla_{g_x} \varphi}^2_{L^2(g_x)} = C \norm{\delbar_x \varphi}^2_{L^2(g_x)} \ge C \norm{\varphi}^2_{L^2(g_x)}.
\]
Let now $\varphi + \xi= Q_x(\psi)$, where $\xi$ is the $L^2(g_x)$-orthogonal projection of $\psi$ on $\mf{k}$. 
It follows that
\begin{align*}
\norm{Q_x\psi}_{L^{2}(g_x)} \le \norm{\varphi}_{L^{2}(g_x)}+ \norm{\xi}_{L^{2}(g_x)} \le C \norm{\psi}_{L^{2}(g_x)} + \norm{\psi}_{L^{2}(g_x)} \le (C+1)\norm{\psi}_{L^{2}(g_x)}.
\end{align*}
Combining this with the Schauder estimate
\[
\norm{\varphi}_{W^{2, p+4}(g_x)} \le C \left( \norm{\psi}_{L^2(g_x)} + \norm{\psi}_{W^{2,p}(g_x)} \right),
\]
we obtain
\[
\norm{Q_x\psi}_{W^{2,p}(g_x)}\le C\norm{\psi}_{W^{2,p}(g_x)},
\]
where again the constants can be made independent of $x$.
This proves the inequality \eqref{eq:rinverse_lich}.

Since
\[
\begin{split}
\Scal (\omega, \Phi(x)) = \widehat S + O(\card{x}),\quad
\vartheta_x = O(\card{x}), \quad
N_x = O(\card{x}^2),
\end{split}
\]
then $L_x(\varphi)$ differs from the acual linearisation $DG_x \vert_{(\xi,0)}(\varphi)$ by a term which is $O(\card{x})$, so
the same estimates \eqref{eq:rinverse_lich} holds for the right inverse to $DG_x \vert_{(\xi,0)}$, up to shrinking $V$.
Therefore by the implicit function theorem for each $x$ we can find $(\xi, \varphi_x)$, varying smoothly on $x$, such that $G_x(\xi, \varphi_x) = 0$.
\end{proof}

We will refer to $V$ as the \emph{Kuranishi space} and to $\Phi$ as the \emph{Kuranishi map}.
Since we also allow non-integrable almost complex structures, the slice $V$ is an actual ball. Instead, in the original work by Kuranishi, the set he constructs parametrises only \emph{integrable} complex structures, hence it is a complex analytic subspace $V^{\mrm{int}}$ of our $V$.


\begin{remark}
The above proof of Theorem \ref{Thm:Kuranishi} gives a different perturbation of the K\"ahler structure than the one given by Sz{\'e}kelyhidi \cite{Szekelyhidi_deformations} and Br\"onnle \cite{Bronnle_PhDthesis}, to ensure that the scalar curvature belongs to the Lie algebra of $K$.
In particular, they fix the K\"ahler form $\omega$ and perturb the Kuranishi map, which leads them to differentiate the function $\Scal(\omega, \Phi(x))$ in the $x$-variable. We instead perturb the K\"ahler form, thus the Kuranishi map remains holomorphic, and we differentiate the scalar curvature in the direction of the K\"ahler potential.
\end{remark}

Until this point, we have considered deformations of the complex structure of a fixed symplectic manifold as $(0,1)$-forms with values in the $(1,0)$-tangent bundle.
The point of view which we now adopt following Dervan--Hallam \cite{DervanHallam}, consists of considering deformations of the manifold $M$ as a smooth proper morphism $\m{U} \to V$, called the \emph{universal family}, such that $\m{U}$ is diffeomorphic to $M\times V$ and the morphism is equivariant with respect to the action of $K$.
The fibre over $x \in V$ is given by the almost complex manifold $\m{U}_x= (M, \omega, \Phi(x))$.

Let $\omega_{\m{U}}$ be the pull-back to $\m{U}$ of the 2-form $\omega$ of $M$; $\omega_{\m{U}}$ is symplectic when restricted to the fibres and K\"ahler when restricted to the integrable fibres. The universal family $\m{U}$ admits an almost complex structure $\bb{J}$ such that $\bb{J}\vert_x = \Phi(x)$. The two-form $\omega_{\m{U}}$ is closed and relatively symplectic, and there exists a moment map $\tau: \m{U} \to \mf{k}$ with respect to $\omega_{\m{U}}$: even if $\omega_{\m{U}}$ is only relatively symplectic, one can still define the moment map by requiring that $\tau$ is equivariant and that for each $\xi \in \mf{k}$
\[
\dd\langle \tau, \xi \rangle = \omega_{\m{U}}(\cdot, \sigma_\xi),
\]
where $\sigma_\xi$ is the infinitesimal vector field associated to $\xi$ \cite[Definition 3.3]{DervanHallam}.

The relatively symplectic form $\omega_{\m{U}}$ induces the two-form on $V$
\begin{equation}\label{Eq:WPmetric}
\Omega_{WP} =  \frac{\widehat{S}}{n+1}\int_{\m{U}/V}\omega_{\m{U}}^{n+1} - \int_{\m{U}/V} \rho_{\m{U}} \wedge \omega_{\m{U}}^n,
\end{equation}
where $\rho_{\m{U}}$ is the curvature  of the Hermitian connection induced by $\omega_{\m{U}}$ on the top wedge power of the vertical tangent bundle.
It follows form Fujiki--Schumacher \cite[\S 7]{FujikiSchumacher_moduli_cscK} that $\Omega_{WP}$ is a K\"ahler metric on $V$, called the \emph{Weil--Petersson metric}. In fact, the Weil-Petersson metric is the pull-back via $\Phi$ of the K\"ahler metric on $\scr{J}$ induced by the Hermitian inner product
\begin{align*}
\langle \alpha, \beta \rangle_J := \int_M \langle \alpha, \beta \rangle_{g_J} \frac{\omega^n}{n!}, \quad \alpha, \beta \in T_{J}\scr{J},
\end{align*}
which is the Donaldson-Fujiki K\"ahler metric on $\scr{J}$ \cite{Donaldson_momentmap,Fujiki_momentmap} (see also \cite[\S6.1]{Szekelyhidi_book}).
The following interpretation of the scalar curvature as a moment map is due to Dervan--Hallam {\cite[Theorem 4.6]{DervanHallam}}.

\begin{theorem}\label{Thm:DervanHallam}
There exists a moment map $\mu$ for the $K$-action on $V$ with respect to $\Omega_{WP}$, which, on the smooth locus of $V^{\mrm{int}}$, is given as the projection
\begin{equation}\label{Eq:Moment_map_DH}
\begin{aligned}
\langle \mu(x), \xi \rangle = \int_{\m{U}_x} \langle\tau, \xi \rangle\vert_x \left(\Scal(\omega, \Phi(x))-\widehat{S}\right) \omega^n,
\end{aligned}
\end{equation}
for $\xi \in \mf{k}$.
\end{theorem}
There is a difference in sign with \cite{DervanHallam}: this is due to the fact that we take the opposite sign for the infinitesimal action \eqref{Eq:map_P}.
Moreover, since the smooth locus of $V^{\mrm{int}}$ is dense in $V^{\mrm{int}}$, and the (Chern) scalar curvature is defined for all $x \in V$, the expression \eqref{Eq:Moment_map_DH} actually holds on all $V^{\mrm{int}}$.

We now change $(\omega, J_x)$ to $(\omega_{\varphi_x}, J_x)$ using Theorem \ref{Thm:Kuranishi} and, since $\varphi_x$ vary smoothly with $x$, we also change $\omega_{\m{U}}$ to $\omega_{\m{U}, \varphi}$ globally on $\m{U}$.
Consider the function $\m{M}_{\omega_{\m{U}}, \varphi} : V \to \bb{R}$ defined as
\[
\m{M}_{\omega_{\m{U}}, \varphi}(x) = \int_{0}^{1}\int_{\m{U}_x} \varphi_x \left(\Scal(\omega_{t\varphi_x}, \Phi(x))-\widehat{S}\right)\omega_{t\varphi_x}^n \dd t,
\]
where $(\omega_{t\varphi_x}, J_x) =( \omega + \frac{1}{2}(\dd\dd^c(t\varphi_x))^{1,1}, J_x)$.
This is a smooth function on $V$, and, when $x \in V^{\mrm{int}}$, $\m{M}_{\omega_{\m{U}}, \varphi}(x)$ coincides with (minus) the Mabuchi functional of $\varphi_x$ on the corresponding fibre $\m{U}_x$.
On the integrable locus, where both $\omega_{\m{U}}$ and $\omega_{\m{U}, \varphi}$ are relatively K\"ahler, the expression \eqref{Eq:WPmetric} defines the Weil-Petersson metrics $\Omega_{WP}$ and $\Omega_{WP, \varphi}$.
Then we have the following relation \cite[\S 7]{PhongRossSturm}, \cite[Proposition 4.2]{DervanNaumann_ModuliCscK}.

\begin{proposition}\label{Prop:WP_Mabuchi}
On the smooth locus of $V^{\mrm{int}}$,
\[
\Omega_{WP}-\Omega_{WP, \varphi} = -i\del\delbar\m{M}_{\omega_{\m{U}}, \varphi}.
\]
\end{proposition}

Since the function $\m{M}_{\omega_{\m{U}}, \varphi}$ is a smooth function defined on all of $V$, we can use Proposition \ref{Prop:WP_Mabuchi} to extend the Weil-Petersson metric obtained using $\omega_{\m{U},\varphi}$. In fact, although the expression \eqref{Eq:WPmetric} is only valid on the integrable locus, we define
\[
\Omega_{WP, \varphi} = \Omega_{WP} +i\del\delbar\m{M}_{\omega_{\m{U}}, \varphi},
\]
which is a K\"ahler metric defined on all of $V$ (possibly after shrinking $V$).
Moreover, in the proof of Theorem \ref{Thm:Kuranishi}, we can restrict ourselves to work only with $K$-invariant potentials, so that the $K$-equivariant implicit function theorem guarantees that $\varphi$ is $K$-invariant; hence $\m{M}_{\omega_{\m{U},\varphi}}$ is $K$-invariant. Thus the moment map defined by Theorem \ref{Thm:DervanHallam} with respect to the Weil-Petersson metric $\Omega_{WP, \varphi}$, which has the expression \eqref{Eq:Moment_map_DH} on $V^{\mrm{int}}$, extends continuously to a moment map $\mu_{\varphi}$ on all $V$ as
\[
\langle\mu_{\varphi}, \xi \rangle = \langle\mu, \xi \rangle + \dd^c\m{M}_{\omega_{\m{U}, \varphi}}(\sigma_{\xi}).
\]
The extension is unique up to elements in the centre of $\mf{k}$.
This is the moment map we will work with for the rest of this section and in Appendix \ref{appendix}; in the following, we denote the moment map $\mu_{\varphi}$ again by $\mu$, making the change in the potential $\varphi$ implicit.

Let us denote by
\begin{enumerate}
\item $\mf{k}^M$ the realisation of $\mf{k}$ as the space of holomorphy potentials with respect to $(\omega,J)$;
\item $\mf{k}_\varphi^M$ the realisation of $\mf{k}$ as the space of holomorphy potentials with respect to $(\omega+ i\del\delbar\varphi, J)$.
\end{enumerate}
Moreover, for each $\xi \in \mf{k}$, we denote the induced vector fields $\sigma_\xi^{\m{U}}$ on $\m{U}$, $\sigma_\xi^{M}$ on $M$ and $\sigma_\xi^{V}$ on $V$.
Combining Theorem \ref{Thm:Kuranishi} with Lemma \ref{Thm:DervanHallam}, we obtain the following result, whose proof is identical to the one of \cite[Lemma 4.15]{OrtuSektnan_ssvb}, but we explicitly adapt it here to the context of cscK metrics for sake of completeness.

\begin{proposition}\label{prop:projection_moment_map_cscK}
A point $x \in V^{\mrm{int}}$ is a zero of the moment map $\mu$ if and only if the scalar curvature $\Scal(\omega+i\del\delbar\varphi_x, \Phi(x))$ is constant.
\end{proposition}
\begin{proof}
Let $\mrm{pr}_M : \m{U} \to M$ be the projection. 
From the second item of Theorem \ref{Thm:Kuranishi}, for every $x \in V^{\mrm{int}}$ we can change $\omega$ by a K\"ahler potential $\varphi_x$ with respect to the complex structure $\Phi(x)$, in such a way that
\begin{equation}\label{eq:scal_deformation_ptwise}
\mrm{Scal}(\omega+i\del\delbar{\varphi_x}, \Phi(x)) - \widehat S \in \mf{k}^M_{\varphi_x}
\end{equation}
Since $\varphi_x$ varies smoothly with $x$, it defines a function $\varphi$ on $\m{U}$, so the condition \eqref{eq:scal_deformation_ptwise} can be rephrased as
\begin{equation*}
\mrm{Scal}((\omega_{\m{U},\varphi}, \bb{J})\vert_{\m{U}_x}) - \widehat S \in \mrm{pr}_M^*(\mf{k}^M_{\varphi_x}).
\end{equation*}

Next, we define two immersions of $\mf{k}$ in $C^\infty(M)$:
\begin{align*}
\iota_{1,x} (\xi) =& (h^{\m{U}}_\xi + \sigma_\xi^{\m{U}}(\varphi)) \vert_{U_x} \\
\iota_{2,x} (\xi) =& h_\xi + \sigma_\xi^{M}(\varphi_x) .
\end{align*}
The two immersions are injective linear maps, hence their images are isomorphic to $\mf{k}$.
We can rephrase the condition \eqref{eq:scal_deformation_ptwise} as
\[
\Scal (\omega + i\del\delbar\varphi_x, \Phi(x))- \widehat S \in \iota_{2,x}(\mf{k}).
\]

On the other hand, Theorem \ref{Thm:DervanHallam} implies that a point $x \in V^{\mrm{int}}$ is a zero of the moment map $\mu$ if and only if the projection of the scalar curvature of $(\omega+i\del\delbar\varphi_x, \Phi(x))$ is $L^2(g_x)$-orthogonal to $\iota_{1,x}(\mathfrak{k})$ in $C^{\infty}(M)$.
Denoting by $\Pi_{1,x}$ said orthogonal projection, Theorem \ref{Thm:DervanHallam} implies that $\mu(x)=0$ if and only if
\[
\Pi_{1,x} (\Scal (\omega + i\del\delbar\varphi_x, \Phi(x))- \widehat S) = 0.
\]

To conclude the proof, we must show that $\Pi_{1,x}$ restricted to $\iota_{2,x}(\mf{k})$ is an isomorphism, from which it follows that $\Pi_{1,x} (\Scal (\omega + i\del\delbar\varphi_x, \Phi(x))- \widehat S) = 0$ if and only if $\Scal (\omega + i\del\delbar\varphi_x, \Phi(x))- \widehat S=0$.
To see this, we write
\[
\iota_{1,x}(\xi) = \iota_{2,x}(\xi) + \sigma_\xi^V(\varphi) \vert_{\m{U}_x}.
\]
Let $\{ \xi_1, \dots, \xi_d\}$ be a basis of $\mf{k}$ such that the corresponding holomorphy potentials on $M$, denoted $\{ h_1, \dots, h_d \}$, are $\omega$-orthonormal.
By using this basis for $\mf{k}$, the restriction of $\Pi_{1,x}$ to $\iota_{2,x}(\mf{k})$ can be written as a linear map from $\bb{R}^d$ to itself as 
\[
(\lambda_1, \ldots, \lambda_d) \mapsto (c_1, \ldots, c_d),
\]
where
\begin{align*}
c_j =& \sum_{i=1}^d \lambda_i \left( \frac{\int_M \iota_{2,x} (\xi_i) \cdot \iota_{1,x} (\xi_j)  (\omega+i\del\delbar \varphi_{x})^{n}}{\|\iota_{1,x} (\xi_j)  \|_{L^2(\omega+i\del\delbar \varphi_{x})} } \right) .
\end{align*}
Since $\sigma^V_{\xi}$ vanishes at $x=0$ and $\varphi = O(|x|)$, we can write
\[
 \iota_{2,x} (\xi_i) \cdot \iota_{1,x} (\xi_j) =  \iota_{2,x} (\xi_i) \cdot \iota_{2,x} (\xi_j)  + O(|x|^2).
\]
Similarly, the contribution to $\iota_{2,x}(\xi)$ coming from $\varphi$ is of order $|x|$, so
\[
 \iota_{2,x} (\xi_i) \cdot \iota_{1,x} (\xi_j) =  h_i \cdot h_j + O(|x|).
\]
Therefore we have
\begin{align*}
\frac{\int_M \iota_{2,x} (\xi_i) \cdot \iota_{1,x} (\xi_j)  (\omega+i\del\delbar \varphi_{x})^{n}}{\|\iota_{1,x} (\xi_j)  \|_{L^2(\omega+i\del\delbar \varphi_{x})} } = \int_M h_i \cdot  h_j \omega^n  + O(|x|).
\end{align*}
Since the $ h_i$ form an orthonormal basis with respect to $\omega^n$, we obtain that
\begin{align*}
c_j =& \sum_{i=1}^d \delta_{j}^i \lambda_i + O(|x|) ,
\end{align*}
which is a perturbation of the identity map of $\mf{k}$. Thus for all sufficiently small $x$, $\Pi_{1,x}$ is also an isomorphism.
\end{proof}

Let $\omega_{\m{U}}$ be the relatively Hermitian metric on $\m{U}$ after the change given by the potential $\varphi$ coming from Theorem \ref{Thm:Kuranishi}.
We next compute an expansion of the induced moment map $\mu$ about the origin of $V$.
By definition,
\[
\dd_{x}\langle \mu, \xi \rangle (v) = \Omega_{x}(v, \m{L}_{\sigma_\xi}x),
\]
where we go back to intepreting $x \in V$ as $(0,1)$-form with values in the holomorphic tangent bundle to make sense of the Lie derivative.

The origin of $V$ is a fixed point of the action.
By identifying $T_0V$ with $\widetilde{H}^1$, we consider on $\widetilde{H}^1$ the linear symplectic form 
\begin{equation}\label{Eq:linear_metric_Htilde1}
\Omega_0 (\cdot, \cdot) = \bm{\Omega}_{J_0}(\dd_0\Phi \cdot, \dd_0 \Phi \cdot),
\end{equation}
and the linear action of $K$ induced by the one on $V$. For any $\xi \in \mf{k}$, consider the endomorphism of $\widetilde{H}^1$
\[
A_{\xi}(t) = \dd_0\left(y \mapsto \mrm{exp}(t\xi)\cdot y \right),
\]
where by $\mrm{exp}(t\xi)$ we denote the 1-parameter subgroup of $K$ defined by the element $\xi \in \mf{k}$. It corresponds via $\Phi$ to the flow of the Hamiltonian vector field $\sigma_\xi$ on $M$.
The operator $A_\xi(t)$ is a unitary operator, since it is linear and symplectic, because the group $K$ acts by symplectomorphisms on $V$. Let $A_\xi$ be the skew-hermitian endomorphism of $(\widetilde{H}^1, J_0)$
\begin{equation} \label{Eq:linearised_infinitesimal_action}
A_\xi := \left.\frac{\dd}{\dd t}\right\vert_{t=0} A_\xi(t).
\end{equation}
\begin{definition}\label{Def:map_nu}
We define a map $\nu : \widetilde{H}^1 \to \mf{k}$ by
\[
\langle \nu(v), \xi \rangle = \frac{1}{2}\Omega_0(A_\xi v, v).
\]
\end{definition}

The map $\nu$ can be characterised as a moment map and is related to the scalar curvature \eqref{Eq:Moment_map_DH} as follows \cite[\S 3]{Inoue_ModuliSpaceFano}:
\[
\begin{aligned}
\left.\frac{\dd^2}{\dd t^2}\right\vert_{t= 0} \langle \mu(tv), f \rangle = \langle \nu(v), f \rangle.
\end{aligned}
\]

\section{Optimal symplectic connections}\label{Sec:OSC}

Let $(Y, H_Y)\to (B,L)$ be a holomorphic submersion with a relatively ample line bundle $H_Y \to Y$.
We make the following assumptions to restrict the class of admissible fibrations to those whose fibres satisfy a stability property defined in terms of K-stability.
More precisely we assume that:
\begin{enumerate}
\item\label{Assumption1} the fibres $Y_b$ are analytically K-semistable, which means by definition that they each admit a degeneration to a cscK manifold $X_b$. We also assume that the degeneration is compatible with the fibration structure in the following sense: there exists a holomorphic map $\varpi:(\m{X}, \m{H}) \to (B, L) \times S$, parametrised by a disk $S$, such that for $s\ne 0$, the family $(\m{X}_s, \m{H}_s) \to B$ is isomorphic to the original fibration $\pi_Y:(Y, H_Y) \to B$ and the central fibration at $s=0$ is a family $\pi_X:(X, H_X) \to B$ whose fibres are cscK;
\item\label{Assumption2} the automorphism groups $\mrm{Aut}_0(X_b, H_b)$ of the fibres are all isomorphic.
\end{enumerate}
The first hypothesis is a stability assumption. We will refer to the submersion $X \to B$ as the \emph{relatively cscK degeneration} of $Y\to B$.
The second hypothesis holds if and only if the spaces $H^0(X_b, T^{1,0}X_b)$ are isomorphic as Lie algebras.

A relative version of Ehresmann's theorem \cite[Proposition 4.5]{Ortu_OSCdeformations} implies that $X$ and $Y$ are diffeomorphic. Let $M$ denote the underlying smooth manifold.
Since the Chern classes are integral classes we have that $c_1(H_X)$ coincides with $c_1(H_Y)$ as cohomology classes on $M$. Since $Y$ is a small deformation of $X$, the cohomology class $c_1(H_X)$ is of type $(1,1)$ on $Y$, so $\omega$ is a $(1,1)$-form with respect to the complex structure of $Y$ \cite[\S6.1]{Huybrechts_ComplexGeometry}.
By Moser's theorem \cite[Theorem 7.2]{DaSilva_Lectures_symplectic_geometry}, we can modify the complex structure of $Y$ by a small diffeomorphism so that $\omega$ restricted to the fibres of $\pi_Y$ is compatible with the restriction of the complex structure. Thus we can assume that $\omega$ is relatively K\"ahler on $Y$.

Therefore we can view $Y\to B$ and $X\to B$ as the same relatively symplectic fibration $(M, \omega) \to B$ with two different integrable almost complex structures $J$ and $I$
where $(\omega, I)$ is relatively cscK and $(\omega, J)$ is just relatively K\"ahler. The family $\m{X} \to B\times S$ corresponds to a family of complex structures $\{J_s\}$ on $(M, \omega) \to B$, such that for $s \ne 0$, $J_s$ is isomorphic to $J$ and $J_0$ is isomorphic to $I$.
In particular, for each $k \gg0$ we have a family of K\"ahler metrics
\[
(\omega_k:=\omega + k\omega_B, J_s),
\]
which are all isomorphic for $s \ne0$.

\subsection{Splitting of the function space}
The relatively K\"ahler form $\omega$ induces a splitting of the tangent space into a vertical and horizontal space
\[
TX = \m{V} \oplus \m{H}^\omega,
\]
where the vertical tangent space is given by the tangent space at every fibre and the horizontal tangent space is defined as the $\omega$-orthogonal space to $\m{V}$. In the language of symplectic geometry, $\omega$ is called a \emph{symplectic connection}.
This splitting extend to all tensor bundles of $TX$ and $T^*X$.

Under the relatively cscK assumption, we now explain how the function space $C^\infty(X, \bb{R})$ splits in a way that takes into account the fibration structure.
We begin by considering the vertical Lichnerowicz operator,
\begin{equation*}
\m{D}_{\m{V}}^*\m{D}_{\m{V}} : C^\infty(X, \bb{R}) \to C^\infty(X, \bb{R}),
\end{equation*}
defined fibrewise as $\left( \m{D}_{\m{V}}^*\m{D}_{\m{V}} \varphi \right)|_{X_b} = \m{D}_b^*\m{D}_b \left. \varphi\right|_{X_b}$. It is a real operator since the fibrewise metric is cscK. By integrating a function $\varphi \in C^\infty(X, \bb{R})$ over the fibres of $\pi$, we define a projection
\begin{equation*}
\begin{aligned}
C^\infty(X, \bb{R}) &\longrightarrow C^\infty(B, \bb{R}) \\
\varphi &\longmapsto \int_{X/B} \varphi \omega^m.
\end{aligned}
\end{equation*}
Its kernel is given by the space $C^\infty_0(X, \bb{R})$ of functions that have fibrewise mean value zero. A key step in the study of optimal symplectic connections is that we can further split this space as follows.

Consider the real vector bundle $E \to B$ \cite[\S3.1]{DervanSektnan_OSC1}, whose fibre over $b \in B$ is the real finite-dimensional vector space $\mrm{ker}_0(\m{D}_b^*\m{D}_b)$ of holomorphy potentials on the fibre $X_b$ with mean-value zero with respect to $\omega_b$. It is well defined as a vector bundle since we assume that the complex dimension of the space $H^0(X_b, T^{1,0}X_b)$ of holomorphic vector fields on $X_b$ is independent of $b$ \cite[\S2.3]{Hallam_geodesics}.
The space of smooth global sections of $E$, denoted by $C^\infty(E)$, is given by the kernel over the fibrewise mean-value-zero functions of the vertical Lichnerowicz operator $\m{D}_{\m{V}}^*\m{D}_{\m{V}}$. 
In \cite[Lemma 2.7]{Hallam_geodesics}, Hallam used the Cartan decomposition for the space $\mf{h}(X_b)$ of holomorphic vector fields of the fibre to show that $E_b$ can be also viewed as the vector space of all K\"ahler potentials $\varphi_b$ on $X_b$ of mean-value zero for which $\omega_b + i\del\delbar\varphi_b$ is still cscK.

We can split $C^\infty_0(X)$ as
\begin{equation*}
C^\infty_0(X, \bb{R}) = C^\infty(E) \oplus C^\infty(R),
\end{equation*}
where $C^\infty(R)$ is the fibrewise $L^2$-orthogonal complement with respect to the fibre metric $\omega_b$, i.e.\ for all $\varphi \in \mrm{ker}_0 \m{D}_b^*\m{D}_b$, $\psi \in C^\infty(R)$
\begin{equation*}
\langle \varphi , \psi \rangle_b := \int_{X_b} \varphi \psi \omega_b^m = 0.
\end{equation*}
So we obtain
\begin{equation}\label{Eq:splitting_function_space}
C^\infty(X, \bb{R}) = C^\infty(B) \oplus C^\infty(E)\oplus C^\infty(R).
\end{equation}
We denote by $p_E : C^\infty(X) \to C^\infty(E)$ the projection.
\begin{definition}
We denote by $\m{K}_E$ the space of functions $\varphi \in C^\infty(X)$ such that $\omega + i \del\delbar\varphi$ is still a \emph{fibrewise cscK metric}.
\end{definition}
In particular, if we change the relatively cscK metric $\omega$ to $\omega + i\del\delbar\varphi$ with $\varphi \in \m{K}_E$, the vector bundles $E(\omega)$ and $E(\omega+i\del\delbar\varphi)$ are isomorphic.

\subsection{Optimal symplectic connections}\label{Subsec:OSC}
The definition of optimal symplectic connections involves various curvature quantities and a description of the complex structure of $Y$ as a deformation of the one on $X$.
To describe the deformations of the complex structure $I$, we consider the space $\scr{J}_\pi$ of almost complex structures compatible with $\omega$ and such that $\dd\pi \circ J = J_B \circ \dd\pi$.
The tangent space at $I$ to $\scr{J}_\pi$ can be identified with
\[
T_{I}^{0,1}\scr{J}_\pi = \set*{\left.\alpha \in \Omega^{0,1}(\m{V}^{1,0}) \right| \omega_F (\alpha\cdot, \cdot) + \omega_F (\cdot, \alpha\cdot) = 0},
\]
where $\omega_F$ is the vertical part of $\omega$.
Consider the map
\begin{equation*}
\begin{aligned}
P_{\m{V}}:C^\infty_0(X, \bb{R}) &\longrightarrow T_{I}^{0,1}\scr{J}_\pi \\
\varphi &\longmapsto \delbar (\grad^{\omega_F}\varphi)^{1,0},
\end{aligned}
\end{equation*}
which is the relative version of the map \eqref{Eq:map_P}.
Let $\widetilde{H}^1_{\m{V}}$ be the kernel of the elliptic \cite[\S4.2]{Ortu_OSCdeformations} operator
\[
\square_{\m{V}} = P_{\m{V}}P_{\m{V}}^* + (\delbar^*\delbar)^2,
\]
where the adjoint is computed with respect to any K\"ahler metric on $X$ which restricts to $\omega_F$ vertically.
The space $\widetilde{H}^1_{\m{V}}$ is the space of integrable first-order deformations of $I$. Let $K_\pi$ be the group of biholomorphisms of $I$ which restrict to an isometry on each fibre, with respect to the fibrewise metric defined by $(\omega, I)$:
\[
K_\pi := \mrm{Isom}(\pi_x, \omega) = \set*{f \in \mrm{Diffeo}(X)\mid f^*\omega = \omega \ \text{and} \ \pi_X \circ f = \pi_X}.
\]
The group $K_\pi$ acts on $\widetilde{H}^1_{\m{V}}$ by pull-back.
The following theorem {\cite[\S4.2]{Ortu_OSCdeformations}} is a fibrewise version of Theorem \ref{Thm:Kuranishi}.

\begin{theorem}\label{Thm:relative_Kuranishi}
There exists a neighborhood of the origin $V_\pi \subset \widetilde{H}^1_{\m{V}}$ and a $K_\pi$-equivariant holomorphic map
\begin{equation}\label{Eq:rel_Kuranishi_map}
\Phi : V_\pi \rightarrow \scr{J}_\pi
\end{equation}
such that $\Phi(0) = I$ and
\begin{enumerate}
\item If $v_1, v_2 \in V_\pi$ and $v_1 \in K_b^{\bb{C}}\cdot  v_2$, and if $\Phi (v_1)$ is integrable, then $\Phi(v_1)$ and $\Phi(v_2)$ are isomorphic;
\item For any $J \in \scr{J}_\pi$ integrable close to $I$, there exists $J'$ in the image of $\Phi$ such that $J'$ is isomorphic to $J$;
\item For each $x \in V_\pi$ there is a relatively K\"ahler metric $\omega_x$ such that
\[
\Scal_{\m{V}}\left(\omega_x, \Phi(x)\right) \in C^\infty(E, I).
\]
\end{enumerate}
\end{theorem}

We now describe the curvature quantities determined by $\omega$ involved in the definition of optimal symplectic connections:
\begin{enumerate}
\item the \emph{symplectic curvature} is a two-form on $B$ with values in the fibrewise Hamiltonian vector fields defined for $v_1, v_2 \in \mf{X}(B)$ as
\begin{equation*}
F_{\m{H}}(u_1, u_2) = [u_1^\sharp, u_2^\sharp]^{\mrm{vert}},
\end{equation*}
where $u_j^\sharp$ denotes the horizontal lift. Let $\gamma$ be the map which associates to a fibrewise Hamiltonian vector field its fibrewise Hamiltonian function with fibrewise mean value zero.
Thus we consider $\gamma^*(F_{\m{H}})$, which is a two-form on $B$ with values in $C^\infty_0(Y, \bb{R})$, and we pull it back to $Y$;
\item the curvature $\rho$ of the Hermitian connection induced by $(\omega, I)$ on the top wedge power $\ext{m}\m{V}$. We will primarily consider its purely horizontal part $\rho_{\m{H}}$;
\item the curvature of the deformation family, given by a global section $\nu$ of the vector bundle $E \to B$ of relatively holomorphy potentials of $X \to B$ defined as
\begin{equation}\label{Eq:map_nu_relative}
\nu(b) = \nu_b(v_b),
\end{equation}
where $\nu_b$ is the map of Definition \ref{Def:map_nu}.
Given $x \in V_\pi$, $v \in \widetilde{H}^1_{\m{V}}$, we will write $\nu_x(v)$
when we want to underline the dependence of the map $\nu$ on the complex structure $\Phi(x)$ and the deformation $v$.
\end{enumerate}
\begin{definition}
\cite[\S3.3]{Ortu_OSCdeformations}
A relatively K\"ahler metric $\omega$ on $Y \to B$ is called an \emph{optimal symplectic connection} if
\begin{equation}\label{Eq:OSC}
p_E \left(\Delta_{\m{V}}(\Lambda_{\omega_B} \gamma^*(F_{\m{H}}))+ \Lambda_{\omega_B}\rho_{\m{H}}\right) + \frac{\lambda}{2}\nu=0,
\end{equation}
for a positive number $\lambda$.
In the following, we will use the notation $\Theta(\omega, J) = \Delta_{\m{V}}(\Lambda_{\omega_B} \gamma^*(F_{\m{H}}))+ \Lambda_{\omega_B}\rho_{\m{H}}$.
\end{definition}

Equation \eqref{Eq:OSC} is a second-order elliptic equation on the vector bundle $E \to B$ \cite[\S5.3]{Ortu_OSCdeformations}.
It arises as the subleading order term in the expansion of the scalar curvature, as shown by the following proposition \cite[Proposition 5.4]{Ortu_OSCdeformations}.
\begin{proposition}\label{Prop:expansion_scalar_curvature}
The scalar curvature of $(\omega_k, J_s)$ admits an expansion
\[
\begin{aligned}
\Scal(\omega_k, J_s) = \Scal_\m{V}(\omega, J_s) + k^{-1}\left( \Scal(\omega_B) + \Delta_{\m{V}} (\Lambda_{\omega_B} \omega_{\m{H}}) + \Lambda_{\omega_B} \rho_{\m{H}}\right) + O\left(k^{-2}\right).
\end{aligned}
\]
The vertical scalar curvature admits an expansion
\begin{equation*}\label{Eq:mu_pi_expansion}
\Scal_{\m{V}}(\omega, J_s) = \mu_\pi(x_s) = \widehat{S}_b + \frac{s^2}{2} \nu_\pi(v) + O\left(s^3\right).
\end{equation*}
Therefore, by choosing $s^2 = \lambda k^{-1}$ for $\lambda >0$ we can combine them to give the single expansion
\[
\Scal(\omega_k, J_s) = \widehat{S}_b + k^{-1}\left( \psi_B + p_E(\Delta_{\m{V}} (\Lambda_{\omega_B} \omega_{\m{H}}) + \Lambda_{\omega_B} \rho_{\m{H}}) + \frac{\lambda}{2} \nu(v) + \psi_R\right) + O\left(k^{-3/2}\right).
\]
\end{proposition}

\begin{remark}
The equation with $\nu=0$ is the condition for an optimal symplectic connection in the sense of \cite{DervanSektnan_OSC1}, where all the fibres are required to be cscK, i.e.\ when $X=Y$.
\end{remark}

The linearisation of the equation at a solution is given by the operator \cite[\S5.3]{Ortu_OSCdeformations}
\[
\widehat{\m{L}} = \m{R}^*\m{R} + \m{A}^*\m{A}
\]
on $C^\infty(E)$, where
\begin{equation}\label{Eq:operator_R}
\m{R}(\varphi_E) = \delbar_B \nabla_{\m{V}}^{1,0} \varphi_E
\end{equation}
and
\begin{equation}\label{Eq:operator_mA}
\m{A}(\varphi_E) = \dd_0\Phi\left(A_{\varphi_E} v\right).
\end{equation}
The adjoint is computed with respect to $\omega_F + \omega_B$. Here $\nabla_{\m{V}}^{1,0} \varphi_E$ is a section of the holomorphic tangent bundle; the vertical part of $\delbar \nabla_{\m{V}}^{1,0} \varphi_E$ vanishes since $\varphi_E \in C^\infty(E)$ and the horizontal part is denoted by the expression \eqref{Eq:operator_R}.
The operator \eqref{Eq:operator_R} can be described as follows \cite[\S4.3]{DervanSektnan_OSC1}: let $\m{D}_k^*\m{D}_k$ be the Lichnerowicz operator with respect to the K\"ahler metric $\omega_k$. It admits a power series expansion in negative powers of $k$:
\begin{equation*}
\m{D}_k^*\m{D}_k = \m{L}_0 + k^{-1}\m{L}_1 + O\left(k^{-2}\right),
\end{equation*}
where $\m{L}_0$ is the \emph{vertical} Lichnerowicz operator $\m{D}_{\m{V}}^*\m{D}_{\m{V}}$. Then for $\varphi, \psi$ fibrewise holomorphy potentials
\begin{equation*}
\int_X \varphi \m{L}_1(\psi) \omega^m \wedge \omega_B^n = \int_X \langle \m{R}\varphi, \m{R}\psi \rangle_{\omega_F+\omega_B} \omega^m\wedge \omega_B^n.
\end{equation*}
This means that the operator $\m{R}^*\m{R}$ can actually be seen as $p_E \circ \m{L}_1$ restricted to $\m{C}^\infty_E(X)$. The kernel of $\m{R}$, thus of $\m{R}^*\m{R}$, consists of fibrewise holomorphy potentials which are global holomorphy potentials on $X$ with respect to $\omega_k$, and it is independent of $k$.
The operator \eqref{Eq:operator_mA} is described as
\begin{equation}\label{Eq:linearised_relative_map_nu}
\langle \dd_v \nu (\m{L}_{\nabla_{\m{V}}{\varphi}}v), \psi \rangle = \int_X \langle \dd_0 \Phi \left(A_{\varphi}v\right), \dd_0 \Phi\left(A_\psi v\right) \rangle_{\omega_F} \omega_F^m \wedge \omega_B^n,
\end{equation}
where $\varphi, \psi \in C^\infty(E)$. A function $\psi \in C^\infty(E)$ is in the kernel of $A$ if and only if $\psi$ is a fibrewise holomorphy potential with respect to all $J_s$, i.e. $\psi \in C^\infty(E, J_s)$.
Thus the kernel of $\widehat{\m{L}}$ consists of those functions $\psi \in C^\infty(E, J_0)$ such that  $\delbar_s (\nabla_{s, \m{V}}^{1,0}\psi )= 0$ for all $s$ \cite[Proposition 5.8]{Ortu_OSCdeformations}.

\subsection{Automorphisms of the optimal symplectic connection equation}
Let $(X, H_X) \to B$ be a relatively cscK fibration. Consider the complex group $\mrm{Aut}(X, H_X)$ of automorphisms of $X$ lifting to $H_X$. Its Lie algebra is given by
the holomorphic vector fields which vanish somewhere, and we denote it by $\mf{h}_0$.
Recall from \S \ref{Subsec:OSC} the group of relative Hamiltonian isometries $K_\pi$.
\begin{definition}
The group of relative automorphisms is
\[
\mrm{Aut}(\pi_X) = \set*{f \in \mrm{Aut}(X, H_X) \mid \pi_X \circ f = \pi_X}.
\]
\end{definition}
We denote by $\mf{h}_\pi$ the Lie algebra of $\mrm{Aut}(\pi_X)$ and $\mf{k}_\pi$ the Lie algebra of $K_\pi$. An element in $\mf{h}_\pi$ is a holomorphic vector field which vanishes somewhere and whose flow lies in $\mrm{Aut}(\pi_X)$, while an element of $\mf{k}_\pi$ is a holomorphic vector field which corresponds to a Killing vector field under the identification of the real tangent bundle $T_{\bb{R}}X$ with the holomorphic tangent bundle $T^{1,0}X$.
The following fibration version of Theorem \ref{Thm:Matsushima_criterion} is a result of Dervan and Sektnan \cite{DervanSektnan_OSC1,DervanSektnan_OSC3}.
\begin{theorem}\label{Thm:autom_OSC}
\begin{enumerate}
\item Let $\omega$ be an optimal symplectic connection on the relatively cscK fibration $X \to B$ and let $f \in \mrm{Aut}(\pi_X)$. Then $f^*\omega$ is an optimal symplectic connection.
\item Let $\omega$ be an optimal symplectic connection on $X \to B$. Then
\[
\mf{h}_\pi = \mf{k}_\pi \oplus I \mf{k}_\pi.
\]
\end{enumerate}
\end{theorem}
In particular, the theorem implies that $K_\pi^{\bb{C}}$ is contained in $\mrm{Aut}(\pi)$ with equality holding if $(\omega, I)$ is an optimal symplectic connection.

We next prove an analogous result for the optimal symplectic connection equation \eqref{Eq:OSC} on a fibration with K-semistable fibres.
Let $(Y, H_Y)\to (B,L)$ be such a fibration admitting a degeneration to $(X, H_X)\to (B,L)$ and let $V_\pi$ be the Kuranishi space of $\pi_X$.
Let $(\m{X}, \m{H})\to (B, L)\times S$ be the degeneration family. The family of complex structures $\{J_s\}$ with $J_0=I$ corresponds to a family $\{y_s\}$ of points in $V_\pi$ such that $x_0$ is the origin of $V_\pi$.
Let $v$ be the tangent vector at the origin of $V_\pi$ that represents the degeneration family, i.e.\
\[
v = \del_s \vert_{s=0} y_s.
\]
Consider the stabiliser of $v$ for the action of $K_\pi$,
\[
K_{\pi,v} := \set*{f \in K_\pi \mid f^*v = v},
\]
and
\begin{equation}\label{Eq:autom_group_genOSC}
G_{\pi,v} := (K_\pi^{\bb{C}})_v.
\end{equation}
For $f \in G_{\pi,v}$
\[
\del_s\vert_{s=0} y_s = v = f^*v = f^*\left(\del_s\vert_{s=0}y_s\right) = \del_s\vert_{s=0} \left(f^*y_s\right).
\]
Therefore
\[
\del_s\vert_{s=0}(y_s - f^*y_s) = 0,
\]
so $v = f^*v$.
So the elements of $G_{\pi,v}$ are automorphisms of the complex structure $I$ of the relatively cscK degeneration $X\to B$ that preserve the projection $\pi_X$ and are also automorphisms of the complex structures $J_s$.
Moreover, the pull-back of the optimal symplectic connection operator via $f\in G_{\pi,v}$ satisfies
\[
f^* \left( \frac{1}{2} \nu (v) + p_E (\Theta(\omega, I))\right) = \frac{1}{2} \nu (v) + p_E (\Theta(f^*\omega, I)).
\]
Indeed, since $\nu$ is $K_\pi^{\bb{C}}$-equivariant,
\[
f^*\nu(v) = \nu (f^*v) = \nu(v),
\]
and by Theorem \ref{Thm:autom_OSC},
\[
f^*(p_E (\Theta(\omega, I))) = p_E (\Theta(f^*\omega, I)).
\]
We have proven the following.
\begin{lemma}
Let $\omega$ be an optimal symplectic connection and $f \in G_{\pi,v}$. Then $f^*\omega$ is an optimal symplectic connection. Moreover, if $\varphi$ is a fibrewise $I$-holomorphy potential whose flow of the gradient lies in $G_{\pi,v}$, $\varphi$ is in the kernel of the linearisation $\widehat{\m{L}}$.
\end{lemma}
Let $\mf{g}_{\pi,v}$ be the Lie algebra of $G_{\pi,v}$, consisting on those holomorphic vector fields whose flow lies in $K_\pi^{\bb{C}}$ and which preserve $v$. In particular, preserving $v$ means that they extend to holomorphic vector fields with respect to all $J_s$.
Let $\mf{k}_{\pi,v}$ be the Lie algebra of $K_{\pi,v}$, of Killing holomorphic vector fields whose flow preserves $v$.
We can then prove a version of Theorem \ref{Thm:Matsushima_criterion} for our setting.

\begin{theorem}
Let $\omega$ be an optimal symplectic connection. Then
\[
\mf{g}_{\pi,v} = \mf{k}_{\pi,v} \oplus I \mf{k}_{\pi,v}.
\]
In particular $K_{\pi,v}$ is a reductive subgroup of $G_{\pi,v}$.
\end{theorem}
\begin{proof}
As recalled in \S\ref{Subsec:OSC}, the kernel $\widehat{\m{L}}$ of the linearisation of the optimal symplectic connection equation consists of fibrewise $I$-holomorphy potentials which are also global $J_s$-holomorphy potentials for all $s$.
From the discussion above, this is in bijection with the Lie algebra $\mf{g}_{\pi,v}$, and $\mf{k}_{\pi,v}$ corresponds to the real vector fields in $\mf{g}_{\pi,v}$. Since $\widehat{\m{L}}$ is a real operator, $\widehat{\m{L}}(u+iv) = 0$ if and only if $\widehat{\m{L}}(u) = 0$ and $\widehat{\m{L}}(v)=0$.
\end{proof}

\subsection{Stability of fibrations}
We recall the notion of fibration degeneration and of fibration stability \cite{DervanSektnan_OSC2, Hallam_geodesics, Hattori_f-stability}.
We then give a notion of stability for a fibration in terms of a numerical invariant associated with fibration degenerations, derived from the adiabatic expansion of the Donaldson-Futaki invariant on the total space.
\begin{definition} 
A \emph{fibration degeneration} for a fibration $(Y, H_Y) \to (B,L)$ is the data of:
\begin{enumerate}
\item a variety $\m{Y}$ together with a flat projective morphism $\m{Y} \to B \times \bb{C}$ with connected fibres and equivariant with respect to a $\bb{C}^*$-action on $\bb{C}$;
\item A $\bb{C}^*$-equivariant line bundle $\m{H} \to \m{Y}$ that is relatively ample over $B \times \bb{C}$;
\item An isomorphism $(\m{Y}_1, \m{H}_1) \simeq (Y, H^r)$ as fibrations over $B$.
\end{enumerate}
A fibration degeneration is called \emph{trivial} if there exists a $\bb{C}^*$-equivariant isomorphism $\m{Y} \simeq Y \times \bb{C}$ with respect to the trivial $\bb{C}^*$-action. A fibration degeneration is called a \emph{product} fibration degeneration if there exists a $\bb{C}^*$-equivariant isomorphism $\m{Y} \simeq Y \times \bb{C}$ where the action is not necessarily trivial.
\end{definition}
A fibration degeneration is essentially a test configuration that preserves the fibration structure.
We next associate a fibration degeneration with a numerical invariant. For each $k \gg 0$ the map $(\m{Y}, \m{H} + kL) \to \bb{C}$ obtained projecting $B \times \bb{C}$ to the second factor is a test configuration for $Y$ with respect to the polarisation $H_Y + kL$.
Thus we expand the Donaldson-Futaki invariant as follows:
\begin{equation}\label{Eq:DF_invariant_expansion}
DF(\m{Y}_k, \m{H} + k L) = W_0 + k^{-1}W_1 + \Ok{-2}.
\end{equation}

\begin{definition}\label{Def:fibration_stability}
We say that the fibration $(Y, H_Y) \to (B, L)$ is
\begin{enumerate}
\item \emph{semistable} if $W_0 \ge 0$ and when $W_0 = 0$ then $W_1 \ge 0$ for all fibration degenerations $(\m{Y}, \m{H}) \to B \times \bb{C}$;
\item \emph{stable} if it is semistable and when $W_0 = 0$ and $W_1 = 0$, then there exists an open $U \subseteq B$ whose complement has codimension at least 2 such that the fibration degeneration normalises to a trivial fibration degeneration over $U$;
\item \emph{polystable} if it is semistable and when $W_0 = 0$ and $W_1 = 0$, then there exists an open $U \subseteq B$ whose complement has codimension at least 2 such that the fibration degeneration normalises to a product fibration degeneration over $U$.
\end{enumerate}
\end{definition}

\section{Stability and optimal symplectic connections}
Let $Y\to B$ be a fibration with K-semistable fibres and let $\m{X} \to B \times S$ be a degeneration family to a relatively cscK fibration $X \to B$, with relatively cscK metric $(\omega,I)$.
The volume of the K\"ahler form $\omega_k$ admits the following expansion in powers of $k$:
\begin{equation}\label{eq:expansions}
\omega_k^{n+m} = k^n \binom{n+m}{n}\omega^m\wedge \omega_B^n + k^{n-1}\binom{n+m}{n-1}\omega^{m+1}\wedge\omega_B^{n-1}+ \Ok{n-2}.
\end{equation}

Let $V_\pi$ be the relative Kuranishi space defined in Theorem \ref{Thm:relative_Kuranishi} and let $y^0$ be the point of $V_\pi$ which corresponds to $J$ via the Kuranishi map $\Phi$ \ref{Eq:rel_Kuranishi_map} and $y^0_s$ be the degeneration family. In particular, the origin of $V_\pi$ corresponds to the relatively cscK fibration $X\to B$.
Let
\[
v = \del_s\vert_{s=0} y^0_s.
\]

The following result \cite[Lemma 3.7]{Ortu_OSCmoduli} describes the subspace of $V_\pi$ which parametrises relatively K-semistable fibrations with smooth fibres admitting a degeneration to a relatively cscK fibration.
\begin{lemma}\label{Lemma:openness_setting}
There exists a locally closed subvariety $V_\pi^+$ of the relative Kuranishi space $V_\pi$ such that, for any $y'\in V^+_\pi$ the fibration $Y'=(M, \omega, \Phi(y'))\to B$ degenerates to a fibration $X'=(M, \omega, I') \to B$ satisfying the following properties:
\begin{enumerate}
\item $(\omega, I')$ is relatively cscK;
\item the groups $\mrm{Aut}(X'_b, H_b')$ are isomorphic for all $b \in B$.
\end{enumerate}
\end{lemma}

\subsection{A finite-dimensional moment map on fibrations}\label{Sec:momentmap}
In this section, we prove that optimal symplectic connections can be described as zeroes of a moment map on $V_\pi^+$: the strategy is the same as that illustrated in \S \ref{Sec:Sz_def_theory} for the cscK equation.
We begin by describing a symplectic form on $V_\pi^+$, defined as follows: given two tangent vectors $\alpha_1, \alpha_2$ at a point $y \in V_\pi^+ \subset V_\pi$, the differential of the Kuranishi map \ref{Eq:rel_Kuranishi_map} is an injective map
\[
\dd_y\Phi : T_yV_\pi \hookrightarrow \widetilde{H}^1_{\m{V}}
\]
inside a space of $(0,1)$-forms with values in the $(1,0)$-vertical tangent bundle of $X \to B$.
Then the Hermitian metric $\omega_F+\omega_B$ induces an Hermitian metric on $T_yV_\pi^+$ whose imaginary part is given by:
\begin{equation}\label{Eq:WP_metric}
\langle \alpha_1,\alpha_2\rangle_{\omega_{F}+\omega_B} = \int_{X} \Lambda_{\omega_{F}+\omega_B}\mrm{Tr}_{\omega_{F}} (\alpha_1\overline{\alpha_2}) \omega^m\wedge\omega_B^n.
\end{equation}

\begin{proposition}[{\cite[Theorem 4.7]{Ortu_OSCmoduli}}]
The 2-form 
\[
\Omega_y (\alpha_1, \alpha_2) = \langle J_y\alpha_1, \alpha_2\rangle_{\omega_F+\omega_B},
\]
where $J_y = \Phi(y)$, is a Weil--Petersson type K\"ahler form on $V_\pi^+$.
\end{proposition}

From Lemma \ref{Lemma:openness_setting}, to any point $y \in V_\pi^+$ we can associate a pair $(x_y, v_y) \in TV_\pi$ where $\Phi(x_y)$ is relatively cscK.
We can then write the optimal symplectic connection equation \eqref{Eq:OSC} as
\begin{equation}\label{Eq:OSCfindim}
p_{E(x_y)}\left( \Theta(\omega, \Phi(x_y))\right) + \frac{1}{2}\nu_{x_y}(v_y)=0.
\end{equation}

\begin{lemma}
The group $K_{\pi,v}$ acts on $V_\pi^+$ by pull-back and the action preserves $V_\pi^+$.
\end{lemma}
\begin{proof}
The claim follows from the fact that the group $K_{\pi, v}$ is a group of biholomorphisms of the complex structure $I$. Consider a point $y \in V_\pi^+$ and a function $f \in K_{\pi,v}$. Then the relatively cscK degeneration $x_y$ is preserved by $f$, i.e.\ $f^*\Phi(x_y)$ is relatively cscK. Therefore the fibration associated with $f^*y$ degenerates to the relatively cscK fibration associated with $f^*x_y$, so $f^*y$ is in $V_\pi^+$.
\end{proof}

For $\xi\in \mf{k}_{\pi,v}$, let $\sigma_\xi$ be the corresponding fibrewise holomorphic vector field and $h_\xi$ be its fibrewise holomorphy potential with respect to the relative K\"ahler metric $(\omega, I)$. If we perturb $\omega$ to $\omega + i\del\delbar\varphi$, for some K\"ahler potential $\varphi \in \m{K}_E(I)$, then \cite[Lemma 4,10]{Szekelyhidi_book} the function
\[
h_{\xi,\varphi} =h_\xi + \sigma_\xi(\varphi)
\]
is a fibrewise holomorphy potential for the same fibrewise holomorphic vector field with respect to the relative metric $(\omega + i \del\delbar \varphi, I)$.
The following result, analogous to the second part of the Kuranishi Theorem \ref{Thm:Kuranishi}, allows us to perturbs the relatively K\"ahler metric in such a way that the optimal symplectic operator has values in the Lie algebra of $K_{\pi, v}$.

\begin{lemma}\label{Lemma:perturbation_Kuranishi_map}
Assume that $Y$ admits an optimal symplectic connection. Then for each $y \in V_\pi^+$ there exists a K\"ahler potential $\varphi$ for the complex structure $\Phi(x_y)$ and $\xi \in \mf{k}_{\pi,v}$, also depending on $y$, such that
\[
p_{E(x_y)}\left( \Theta(\omega + i\del\delbar\varphi, \Phi(x_y))\right) + \frac{1}{2}\nu_{x_y}(\dd_y\Phi(v_y)) = h_{\xi,\varphi}.
\]
\end{lemma}

\begin{proof}
For each $y \in V_\pi^+$ let $g_y$ be the Riemannian metric induced by the volume form $\omega^m \wedge \omega_B^n$ with respect to the complex structure $\Phi(x_y)$.
Consider the operator
\begin{equation*}\label{Eq:OSC_operator}
\begin{aligned}
\m{F}_y : \m{K}_E^{2,p+2}(g_y) \times \mf{k}_{\pi,v} &\to W^{2, p}(g_y) \\
(\varphi, \xi) &\mapsto p_{E(\varphi, x_y)} \left( \Theta(\omega + i \del\delbar\varphi, \Phi(x_y))\right) + \frac{\lambda}{2}\nu_{\varphi, x_y}\left(v_y\right)-h_{\xi, \varphi}.
\end{aligned}
\end{equation*}
In this expression, $E(\varphi, x_y)$ is the vector bundle of fibre holomorphy potentials with respect to the K\"ahler structure $(\omega +i\del\delbar\varphi, \Phi(x_y))$ and $\m{K}_E^{2,p+2}(g_y)$ is the $W^{2,p+2}(g_y)$-completion of $\m{K}_E(\Phi(x_y))$.
As recalled in \S \ref{Subsec:OSC}, the differential of $\m{F}_y$ at $(\xi,0)$ with respect to $\varphi$ is \cite[\S 5.3]{Ortu_OSCdeformations}
\[
L_y(\psi) = -\frac{1}{2}\m{R}_y^*\m{R}_y(\psi) -\m{A}_y^*\m{A}_y(\psi)-h_{\xi}-\sigma_\xi(\psi),
\]
where $\m{R}_y$ and $\m{A}_y$ are the operators \eqref{Eq:operator_R} and \eqref{Eq:operator_mA}, computed with respect to the complex structure $\Phi(x_y)$.
As an operator on the vector bundle $E(x_y) \to B$ of fibrewise $\Phi(x_y)$-holomorphy potentials, the kernel of $L_y$ is contained in $\mf{k}_{\pi,v}$, so it is surjective.

We next prove that $L_y$ has a bounded right inverse.
As a differential operator on the vector bundle $E(x_y)$, the term $-\m{A}_y^*\m{A}_y(\psi)-h_{\xi}-\sigma_\xi(\psi)$ is linear \cite[\S 5.3]{Ortu_OSCdeformations}, so $L_y$ has a bounded right inverse if and only if $\m{R}_y^*\m{R}_y$ does.
As in the proof of Theorem \ref{Thm:Kuranishi}, the bound follows from a Poincar\'e type inequality and Schauder estimates.
Recall that by definition
\[
\m{R}_y(\psi) = \delbar_B \nabla^{1,0}_{\m{V},y}\psi.
\]
The vertical-horizontal splitting of $TX$ induced by $\omega$ gives a splitting of the operator
\[
\dd = \dd_{\m{H}} + \dd_{\m{V}},
\]
and analogously for the holomorphic operators $\del_0$, $\delbar_0$ with respect to the complex structure $\Phi(0)$.
These splittings define two Laplace operators on $X$ \cite[\S 2]{LuSeyyedali_extremal}:
\begin{align*}
\Delta_{H,0} (\psi) &= \delbar_B^*\delbar_B\psi,\\
\Delta_{\m{V},0}(\psi) &= \delbar_{\m{V}}^* \delbar_{\m{V}}\psi,
\end{align*}
where the adjoint is defined with respect to $g_0$, and similarly the Laplace operator $\Delta_{\m{V},y}(\psi) = \delbar_{\m{V},y}^* \delbar_{\m{V},y}$ where the adjoint is again defined with respect to $g_0$ but the vertical $\delbar$-operator is the one induced by the complex structure $\Phi(x_y)$.
Arguing as in the proof of Theorem \ref{Thm:Kuranishi}, we obtain the Poincar\'e inequality for $\delbar_B$,
\[
\norm{\delbar_B\psi}^2_{L^2(g_y)} \ge C \norm{\psi}^2_{L^2(g_y)}, \quad \psi \in \mf{k}_{\pi,v}^\perp \cap C^\infty(E(x_y)),
\]
where $C$ is independent on $y$, for $y$ sufficiently close to the origin.
It remains to give a Poincar\'e type inequality for $\Delta_{\m{V},y}$.
This again follows as in the proof of Theorem \ref{Thm:Kuranishi}: in fact, on each fibre the operator $\Delta_{\m{V}, y}$ is the Laplacian on the cscK fibre of the fibration corresponding to $\Phi(x_y)$, so the inequality
\[
\norm{\nabla^{1,0}_{\m{V}}(\psi)}_{L^2(g_y)} \ge C \norm{\psi}_{L^2(g_y)}
\]
holds at any given fibre. Hence it holds in a neighbourhood of a given fibre with the constant independent on the fibre. By compactness of the total space, it holds globally with a uniform constant.
The conclusion then follows applying the Schauder estimates as in Theorem \ref{Thm:Kuranishi}: the right inverse $Q_y$ of $L_y$ satisfies
\[
\norm{Q_y} \le C.
\]

From the implicit function theorem, locally around $(0, y_0)$ we can then solve $\m{F}_y(\varphi_y, \xi) = 0$, and moreover $\varphi_y$ varies smoothly with $y$.
\end{proof}

Next we wish to prove an analogue of Theorem \ref{Thm:DervanHallam} to show that the optimal symplecic connection operator is a moment map.
Let $\m{U} \to B\times V_\pi^+$ be the universal family of fibrations such that $\m{U}_0 = Y \to B$ admits an optimal symplectic connection.
The fibres of $\m{U} \to V_\pi^+$ are diffeomorphic to $M$, where $M$ is the underlying smooth manifold to the total space $Y$.
Let $\omega_{\m{U},k}$ be a relatively K\"ahler metric on $\m{U}$ obtained pulling back the metric $\omega_k$ from $Y$.
Then
\[
\omega_{\m{U}, k} =\omega_{\m{U}}+k\omega_B,
\]
where $\omega_{\m{U}}$ is the pullback of $\omega$ from $Y$.
The group $K_{\pi,v}$ acts on $\m{U}$ so that the family $\m{U}\to B \times V_\pi^+$ is $K_{\pi, v}$-equivariant, where the action is trivial on $B$.

Since the K\"ahler potentials $\varphi_y$ obtained in Lemma \ref{Lemma:perturbation_Kuranishi_map} vary smoothly with $y$, they define a global potential $\varphi$ for $\omega_{\m{U}}$ that makes the optimal symplectic connection operator lie in the Lie algebra $\mf{k}_{\pi,v}^{M}$, i.e.\ for each $y \in V_\pi^+$,
\[
p_{E(x_y)}\left( \Theta(\omega_y, \Phi(x_y))\right) + \frac{1}{2}\nu_{x_y} \in (\mf{k}_{\pi,v}^M)_y.
\]

Let $\tau$ be a moment map on $\m{U}$ for the action of $K_{\pi,v}$ with respect to $\omega_{\m{U}}$.
As a consequence of Theorem \ref{Thm:DervanHallam}, we obtain the following interpretation of the optimal symplectic connection operator as a moment map.
\begin{theorem}\label{Thm:osc_momentmap}
The map $\theta$ defined as
\begin{equation}\label{Eq:osc_momentmap}
\langle \theta(y), v \rangle = \int_{\m{U}_y} \langle \tau, v \rangle\vert_y \left( p_{E(x_y)}\left( \Theta(\omega_y, \Phi(x_y))\right) + \frac{1}{2}\nu_{x_y}\right) \omega_y^m\wedge \omega_B^n
\end{equation}
is a moment map on $V_\pi^+$ with respect of the Weil--Petersson metric $\Omega $.
\end{theorem}
\begin{proof}
We write an expansion of the expression \eqref{Eq:Moment_map_DH}, in powers of k.
First observe that a moment map on $\m{U}$ with respect to the relatively K\"ahler metric $\omega_{\m{U},k}$ is given by
\[
\tau_k = \tau + k\tau_B
\]
In fact, this follows from expanding in $k$ the expression
\[
\dd \langle \tau_k, v \rangle = \omega_{\m{U}, k}(\cdot, v).
\]
From Theorem \ref{Thm:DervanHallam}, we obtain a moment map $\mu_k$ defined as
\[
\langle \mu_k(y), v \rangle = \int_{\m{U}_y} \langle\tau_k, v \rangle\vert_y \left(S(\omega_{\m{U},k}\vert_y, \Phi(x)) - \widehat{S}_k\right) \omega_k^{n+m}.
\]
Since $v$ is a vertical deformation, $\omega_B (\cdot, v) =0$, so we are left with the expansion of
\[
\langle \mu_k(y), v \rangle = \int_{\m{U}_y} \langle\tau, v \rangle\vert_y \left(S(\omega_{\m{U},k}\vert_y, \Phi(x)) - \widehat{S}_k\right) \omega_k^{n+m}.
\]
Using the expansion of the volume \eqref{eq:expansions} and Proposition \ref{Prop:expansion_scalar_curvature}, we obtain that the leading order term is precisely the expression \eqref{Eq:osc_momentmap}.
\end{proof}

We finish this section by observing that zeroes of the moment map $\theta$ correspond to optimal symplectic connections.
In fact, Theorem \ref{Thm:osc_momentmap} implies $y$ is a zero of the moment map $\theta$ if and only if the optimal symplectic connection operator is orthogonal to the Lie algebra $\mf{k}_{\pi,v}$.
Arguing as in Proposition \ref{prop:projection_moment_map_cscK}, we obtain the following.
\begin{proposition}
A point $y \in V_\pi^+$ is a zero of the moment map $\theta$ if and only if
\[
p_{E(x_y)}\left( \Theta(\omega_y, \Phi(x_y))\right) + \frac{1}{2}\nu_{x_y} =0.
\]
\end{proposition}
The upshot is that we can re-define $\omega_{\m{U}}$ to take into account the change in the K\"ahler potential $\varphi$ and that the zeroes of the moment map $\theta$ obtained from $\omega_{\m{U}}$ are optimal symplectic connections.

\subsection{Optimal symplectic connections on stable deformations}
In this section, we use the moment map property to prove a finite-dimensional stability result analogue to Theorem \ref{Thm:App}: we prove that a stable deformation of a fibration with an optimal symplectic connection still admits an optimal symplectic connection.

As observed in \S \ref{Sec:Sz_def_theory}, Theorem \ref{Thm:osc_momentmap} only works on the subspace of $V_\pi^+$ that parametrises integrable complex structures.
However, we can extend the K\"ahler metric \eqref{Eq:WP_metric} on $V_\pi^+$ to a K\"ahler metric on the whole $V_\pi$, so the optimal symplectic connection moment map $\theta$ extends in a unique way to $V_\pi$.
We do this because $V_\pi$ is smooth, so we can apply to it the theory of the moment map flow.
We will prove at the end that the zero of the moment map we find is in fact in $V_\pi^+$ and it corresponds to an integrable complex structure.

\begin{theorem}\label{Thm:rel_stability_osc}
Let $Y_w \to B$ be a deformation of $Y\to B$, with $w \in V_\pi^+$. Assume that $Y \to B$ admits an optimal symplectic connection and that $Y_w \to B$ is polystable. Then there is a $w_1 \in V_\pi^+$ in the $K_{\pi,v}^{\bb{C}}$-orbit of $w$ such that $\theta(w_1) = 0$.
\end{theorem}

\begin{proof}
Assume that $w$ has discrete stabiliser.
From Proposition \ref{Prop:existence_limit_flow}, the moment map flow of $\theta$ starting at $w$ converges to a point $w_\infty$, which belongs to $V_\pi$ by Proposition \ref{prop:stabilityflow}. 
Assume that $w_\infty \notin K_{\pi,v}^{\bb{C}}\cdot w$. By Proposition \ref{Prop:not_in_orbit} we can find $\xi\ne 0$ in $ \mf{k}_{\pi,v}$ and $\widetilde w \in V_\pi$ such that
\begin{equation}\label{eq:theta_inner_product_positive}
\langle \theta (\widetilde{w}), \xi\rangle\ge 0.
\end{equation}

Let $f$ be the fibrewise holomorphy potential induced by $\xi$ and consider the expansion
\[
\langle \Scal(\omega_k, \widetilde{w})-\widehat{S}, f\rangle_k = \int_{M}( \Scal(\omega_k, \widetilde{w})-\widehat{S})f \omega_k^{n+m}.
\]
Expanding both the scalar curvature and the K\"ahler metric in powers of $k$ by plugging in the expansion of the volume \eqref{eq:expansions} and the expansion of the scalar curvature from Proposition \ref{Prop:expansion_scalar_curvature} we obtain that the leading order term is, up to a constant,
\[
k^n \int_{M} (\widehat{S}_b - \widehat{S})f \omega^m\wedge\omega_B^n
\]
and the sub-leading order term is, up to a constant,
\begin{equation}\label{Eq:expansion_pairing_subleading}
k^{n-1} \int_M \theta(\widetilde{w})f\omega^m\wedge\omega_B^n + k^{n-1}\int_M (\widehat{S}_b - \widehat{S})f \omega^{m+1}\wedge\omega_B^{n-1}.
\end{equation}
We know from Lemma \ref{Thm:DervanHallam} that the scalar curvature is a moment map for the action of $K_{\pi,v}$ (in fact, for a larger group of biholomorphic isometries that contains $K_{\pi,v}$).
It follows from Theoerm \ref{Thm:classical_Futaki_equal_DF} that the pairing
\[
\langle \Scal(\omega_k, \widetilde{w})-\widehat{S}, f\rangle_k
\]
is, up to a constant, equal to the opposite of the Donaldson-Futaki invariant of the test configuration for $Y_w$ generated by the infinitesimal vector field $\sigma_f$ with central fibre $Y_{\widetilde{w}}$.
Since $\xi$ is in $\mf{k}_{\pi, v}$ the induced test configuration is in fact a fibration degeneration.
Comparing the expression \eqref{Eq:expansion_pairing_subleading} with the expansion of the Donaldson-Futaki invariant \eqref{Eq:DF_invariant_expansion} we see that the quantity
\[
\langle \theta(\widetilde{w}), \xi \rangle = \int_M \theta(w_\infty)f \omega^m \wedge \omega_B^n
\]
is, up to a constant, equal to the opposite of subleading order term $W_1$, which is negative since the fibration $Y_w \to B$ is polystable.
More precisely, we have shown that if $Y_w \to B$ is polystable, then
\[
\langle \theta(\widetilde{w}), \xi \rangle \le 0
\]
with equality holding if and only if $Y_{\tilde w} \to B$ is isomorphic to $Y_w \to B$.
Therefore, the inequality \eqref{eq:theta_inner_product_positive} forces $\langle \theta(\widetilde{w}), \xi \rangle = 0$, and from Proposition \ref{Prop:not_in_orbit} we obtain that $\theta(\widetilde w)=0$.

If $w$ has non-discrete stabiliser, we project orthogonally to a maximal torus in the stabiliser as explained in \S \ref{subsec:momentmapflow}.
Let $\mf{k}_{T^\perp}$ be the Lie algebra \eqref{eq:Lie_algebra_orthogonal} defined from $\mf{k}_{\pi,v}$, and $K_{T^\perp}$ the corresponding group.
Arguing as before, we obtain $\widetilde w \in K_{T^\perp}\cdot w$ such that $\theta_{T^\perp}(\widetilde w)=0$.
From \cite[Theorem 5.25]{Szekelyhidi_book}, it follows that there exists $w_1 \in K_{T^\perp}^{\bb{C}}\cdot \widetilde w$ such that $w_1$ is a critical point for $\norm{\theta}^2$. This implies \cite[Lemma 5.23]{Szekelyhidi_book} that $\theta(w_1) \in \left( \mf{k}_{\pi,v} \right)_{w_1}$.
Since $\widetilde w$ is in the same orbit as $w$, so is $w_1$, so the function $\theta (w_1)$ also belongs to the Lie algebra $ \left(\mf{k}_{\pi,v}\right)_{w}$ of the stabiliser of $w$. Therefore stability implies that
\[
\langle \theta(w_1), \theta(w_1) \rangle = 0,
\]
so $\theta(w_1)=0$. Since $V_\pi^+$ is closed under the action of the group, the point $w_1$ is in fact in $V_\pi^+$.
\end{proof}

\appendix
\section{Stability of deformations of cscK manifolds}\label{appendix}

We use the moment map flow technique to prove a result analogous to Theorem \ref{Thm:rel_stability_osc}, in the case of a K\"ahler manifold $(M, \omega, J)$ with constant scalar curvature: K-polystability of a deformation of $M$ implies the existence a cscK metric.
We provide a proof that uses the Dervan--Hallam theory and the moment map flow, analogous to that of Theorem \ref{Thm:rel_stability_osc}.

\begin{theorem}\label{Th:Sz_stability}
Let $(X,L)$ be a cscK manifold and $(Y, L')$ be a K-polystable deformation of $X$. Then $Y$ admits a cscK metric in $c_1(L')$.
\end{theorem}

\begin{proof}
Let us write $X= (M, \omega, J)$, where $(\omega, J)$ is the cscK metric. Let $V$ be the Kuranishi space and $\Phi$ the Kuranishi map, as in \S\ref{Sec:Sz_def_theory}, so $X$ corresponds to the origin of $V$ and $Y$ corresponds to a point $x \in V$.
In particular, $X$ and $Y$ are diffeomorphic and, since the first Chern classes are integral classes, we can assume that $c_1(L) = c_1(L')$.
Let $\m{U} \to V$ be the universal family parametrising the deformations, let $\omega_{\m{U}}$ the relatively K\"ahler metric on $\m{U}$ obtained after applying Theorem \ref{Thm:Kuranishi}, and let $\m{L}$ be the natural relatively ample line bundle on $\m{U}$.

Recall that from Theorem \ref{Thm:DervanHallam} we have a moment map $\mu$ on $V$ with respect to the action of the group $K$ of Hamiltonian isometries, and from Proposition \ref{prop:projection_moment_map_cscK} the zeroes of $\mu$ are cscK metrics.

Assume that $x$ has discrete stabiliser.
We wish to apply the moment map flow
\[
\frac{\dd}{\dd t} x_t = J\sigma(\mu(x_t))
\]
defined in Definition \ref{def:momentmapflow}, with starting point $x \in V$, and show that there exists a zero of $\mu$ in the $K^{\bb{C}}$-orbit of $x$. 
From Proposition \ref{prop:stabilityflow} we know that the flow has a limit $x_\infty$ inside $V$.
From Proposition \ref{Prop:not_in_orbit}, either $\mu(x_\infty) = 0$ and $x_\infty \in K^{\bb{C}}\cdot x$, or $x_\infty \notin K^{\bb{C}}\cdot x$.
In the latter case, let $0\ne\xi \in \mf{k}$ and $\widetilde x\in V$ be given by Proposition \ref{Prop:not_in_orbit} such that
\[
\lim_{t \to \infty} \exp(-it\xi) \cdot x = \widetilde x
\]
and
\begin{equation}\label{eq:pairing_positive}
\langle \mu(\widetilde x), \xi \rangle \ge 0.
\end{equation}

The $\Phi(\widetilde x)$-holomorphic vector field induced by $\xi$ defines a test configuration $(\m{Y}, \m{L}')$ for $(Y, L')$, represented on $V$ by the closure of the orbit of $x$ for the action of the one-parameter subgroup of $K^{\bb{C}}$ induced by $\xi$.
From Theorem \ref{Thm:DervanHallam},
\[
\langle \mu(\widetilde x), \xi \rangle = \int_{\m{U}_{\tilde x}} (\Scal (\omega_{\tilde x}, \Phi(\widetilde x))- \widehat S) f \omega_{\tilde x}^n,
\]
where $f$ is the holomorphy potential with respect to $\Phi(\widetilde x)$ induced by $\xi$. In particular $\langle \mu(\widetilde x), \xi \rangle$ is equal to the Futaki invariant of the vector field $\sigma_\xi$ induced by $\xi$.
By Theorem \ref{Thm:classical_Futaki_equal_DF} this is, up to a constant, equal to the opposite of the Donaldson-Futaki invariant of the test configuration $(\m{Y}, \m{L}')$.
K-polystability of $Y$ then implies that
\begin{equation}\label{eq:pairing_negative}
\langle \mu(\widetilde x),\xi \rangle \le 0,
\end{equation}
with equality holding if and only if the test configuration is a product. 
From \eqref{eq:pairing_positive} and \eqref{eq:pairing_negative} we see that $\langle \mu(\widetilde x),\xi \rangle = 0$. Proposition \ref{Prop:not_in_orbit} gives $\widetilde x = x_\infty$ and $\mu(x_\infty) = 0$, and from stability the test configuration induced by $\xi$ is in fact a product test configuration, which in particular implies that $(\m{U}_{x_\infty}, \m{L}_{x_\infty})$ is isomorphic to $(Y, L')$.
Since $K^{\bb{C}}$ is a group of biholomorphisms and $\Phi(x)$ is an integrable complex structure, so is $\Phi(x_\infty)$.
Hence, from Proposition \ref{prop:projection_moment_map_cscK}, the corresponding K\"ahler metric on $(Y, L')$ is cscK.

If the stabiliser of $x$ is not discrete, we can project the moment map orthogonally to the stabiliser, as explained in \S \ref{subsec:momentmapflow}.
Then arguing as before, we find a zero $x_\infty$ of $\mu_{T^\perp}$.
From \cite[Theorem 5.25]{Szekelyhidi_book}, it follows that there exists $x_1 \in K_{T^\perp}^{\bb{C}}\cdot x_\infty$ such that $x_1$ is a critical point for $\norm{\mu}^2$. This implies \cite[Lemma 5.23]{Szekelyhidi_book} that $\mu(x_1) \in \mf{k}_{x_1}$.
Since $(\m{U}_{x_\infty}, \m{L}_{x_\infty})$ is isomorphic to $(Y, L')$, so is $(\m{U}_{x_1}, \m{L}_{x_1})$, so the function $\mu (x_1)$ also belongs to the Lie algebra $\mf{k}_{x}$ of the stabiliser of $x$. Therefore stability implies that
\[
\langle \mu(x_1), \mu(x_1) \rangle = 0,
\]
so $\mu(x_1)=0$.
\end{proof}

\begin{remark}\label{rmk:error}
Our proof of Theorem \ref{Th:Sz_stability} fixes a mistake in Sz\'ekelyidhi's original proof \cite[Proposition 8]{Szekelyhidi_deformations}.
In fact, Sz\'ekelyidhi's argument relies on a lower bound on the norm of the differential of $\mu$ along the infinitesimal vector field that is uniform in a ball around the origin; this uniform bound cannot exist, because the differential of $\mu$ vanishes at the fixed points, which form a linear subspace. So \cite[Proposition 9]{Szekelyhidi_deformations} cannot be applied.
A similar issue is present also in \cite[Proposition 3.3.2]{TiplerVanCoevering_deformationsSasaki}, where the argument would need a uniform bound on the Hessian of the norm-squared of $\nu$ which cannot hold for the same reason, and in \cite[Theorem 1.27]{Ortu_thesis}.
A related mistake can be found in \cite[\S 2.4]{Bronnle_PhDthesis}, where the techniques require the K\"ahler property but the finite-dimensional reduction gives a moment map with respect to a symplectic form that is not K\"ahler.

While our approach applies generally to perturbation problems, a workaround to this issue was developed in certain cases by Inoue in the case of K\"ahler-Ricci solitons \cite[\S 4.1.2]{Inoue_thesis} and by Fan in the case of Higgs bundles \cite{YueFan_ModuliHiggsBundles} using global techniques specific to those problems.
\end{remark}

The same argument can be applied to the finite-dimensional setting of GIT stability to show the analogous result that the orbit of a polystable deformation of a zero of the moment map contains a zero of the moment map.
This result is used in \cite{Szekelyhidi_deformations} to prove Theorem \ref{Th:Sz_stability}; although our proof of Theorem \ref{Th:Sz_stability} does not rely on finite-dimensional GIT, we report a proof of the finite-dimensional result that follows our strategy, as it is of independent interest.

\begin{theorem}\label{Thm:finite_dim_reduction}
Let $v \in V \subset T_0V$ be a GIT-polystable point for the linearised $G$-action on $T_0V$. Then there is a point $x_0 \in V$ in the same $K^{\bb{C}}$-orbit as $v$ such that $\mu(x_0) = 0$.
\end{theorem}
\begin{remark}\label{rmk:flow_stays_inside}
In the case of GIT stability we can give a different proof of Proposition \ref{prop:stabilityflow} to show the stability of the flow.
Let us compactify $V$ to $X$, where we can think of $X$ as a projective space. Let $S$ be the set of limit points of the moment map flow in $V$. Then
\[
S = \mu^{-1}(0) \cap V \subset \mu^{-1}(0),
\]
where the inclusion is open. Let $\mu^{-1}(0)/K$ be the symplectic quotient. The topological quotient
\[
\mu^{-1}(0) \to \mu^{-1}(0)/K
\]
maps the open set $S$ to an open set $[S]$ in the quotient, since $S$ is $K$-invariant.
From the Kempf--Ness theorem, since the symplectic quotient is isomorphic to the GIT quotient $X^{ss}\sslash G$, the set $[S]$ is open in the GIT quotient, so its preimage is open in $X$.
Now, the preimage of an open set via the GIT quotient is closed under taking the limit of the moment map flow.
Therefore, since $V \cap X^{ss}$ is also open and the preimage of $[S]$ under the GIT quotient is contained in $V \cap X^{ss}$, we find an open set $S'$ inside $V$ of semistable points such that the limit of the flow starting inside $S'$ are in $S'$ (although the flow might leave $S'$ in finite time).
\end{remark}

\begin{proof}[Proof of Theorem \ref{Thm:finite_dim_reduction}]
As before, without loss of generality, we can assume that $v$ has discrete stabiliser, otherwise we can project the moment map orthogonally to the stabiliser, as explained in \S \ref{subsec:momentmapflow}.
From Remark \ref{rmk:flow_stays_inside} we know that the flow of $\mu$ starting at $v$ has a limit $x_\infty$ inside $V$.
Assume that $x_\infty \notin K^{\bb{C}}\cdot v$.
Let $\xi \in \mf{k}$ and $\widetilde v\in V$ be given by Proposition \ref{Prop:not_in_orbit} such that
\[
\lim_{t \to \infty} \exp(-it\xi) \cdot v = \widetilde v
\]
and
\[
\langle \mu(\widetilde v), \xi \rangle \ge 0.
\]
Consider the path $x_t = t\widetilde v$ in $V$; it is well defined because $V$ is a linear space, so $\widetilde v \in T_0V$ can be viewed as an element of $V$ itself. Hence we have the following expansion of $\mu$ along the path $tv$ for small $t$:
\begin{equation}
\mu(t\widetilde v) = \mu(0) +t \dd_0\mu(\widetilde v) + \frac{t^2}{2} \left.\frac{\dd^2}{\dd t^2}\right\vert_{t= 0} \mu(t\widetilde v) + O(t^3).
\end{equation}
By our hypothesis, $\mu(0) =0$. Moreover, the differential $\dd_0\mu(v)$ vanishes, since $\mu$ is a moment map and the origin is a fixed point of the $K$-action, so
\begin{equation}\label{Eq:expansion_mu_nu}
\mu(t\widetilde v) =\frac{t^2}{2}\nu(\widetilde v) + O(t^3).
\end{equation}
The GIT polystability of $v$ implies that, by the Hilbert--Mumford criterion \cite[Theorem 12.2]{GeorgoulasRobbinSalamon_GITbook} the weight of the action on $\widetilde H^1$ of the one-parameter subgroup of $K$ defined by $\xi$
\[
w_{\nu} = - \langle \nu(\widetilde v),\xi \rangle
\]
is nonpositive.
Since $\nu$ is the leading order term of $\mu$, then stability in fact implies that
\[
\langle \mu(\widetilde v),\xi \rangle \le 0.
\]
Therefore, from Proposition \ref{Prop:not_in_orbit} we obtain that $x_\infty \in K^{\bb{C}}\cdot v$.
Therefore, since $v$ has finite stabiliser under $K$, so does $x_\infty$, which implies that $\mu(x_\infty) = 0$.
\end{proof}

\begin{remark}
It follows from the proof of Theorem \ref{Thm:finite_dim_reduction} that the GIT-polystability of $x \in V$ implies K-polystability of the corresponding manifold $(\m{U}_x, \m{L}_x)$.
\end{remark}

\bibliographystyle{acm} 
\bibliography{../bibliografia}

\end{document}